\documentclass[10pt,final,leqno,onefignum,onetabnum]{siamltex}
\usepackage{amssymb,amsfonts,amsmath,bm}
\usepackage{hyperref}
\usepackage{psfrag,graphicx,color}
\usepackage{mathrsfs}
\usepackage{graphics}
\usepackage{subfigure}
\usepackage{enumerate,enumitem}
\usepackage{pgfplots}
\usetikzlibrary{arrows.meta}
\usepackage{MnSymbol}
\usepackage{cite}
\usepackage{diagbox}%

\usepackage[notref,notcite]{showkeys} 

\newtheorem{remark}{Remark}[section]

\definecolor{darkred}{rgb}{0.85,0,0}

\definecolor{green}{rgb}{0,0.7,0}

\def\al{\alpha}
\def\pt{\partial}
\def\dalt{{\pt_t^\al}}
\def\bdalt{{\bar{\pt}_\tau^\al}}
\def\R{{\mathbb R}}
\def\C{{\mathbb C}}

\def\d{{\mathrm d}}
\def\i{{\mathrm i}}

\def\ep{\epsilon}
\def\II{(\Omega)}

\DeclareMathOperator\Real {Re}
\DeclareMathOperator\Imag {Im}

\allowdisplaybreaks[1]
\begin{document}

\title{High-order Time Stepping Schemes for Semilinear Subdiffusion Equations
\thanks{The research of K. Wang is partially supported by a Hong Kong RGC grant (Project No. 15300817),
and that of Z. Zhou by a start-up grant from the Hong Kong Polytechnic University and Hong Kong RGC grant No. 25300818. }}

\author{Kai Wang\thanks{Department of Applied Mathematics, The Hong Kong Polytechnic University, Kowloon, Hong Kong.
(\texttt{kai-r.wang@connect.polyu.hk})}
\and Zhi Zhou\thanks{Department of Applied Mathematics, The Hong Kong Polytechnic University, Kowloon, Hong Kong.
(\texttt{zhizhou@polyu.edu.hk, zhizhou0125@gmail.com})}}
\date{\today}

\maketitle
\begin{abstract}
The aim of this paper is to develop and analyze high-order time stepping schemes for solving semilinear subdiffusion equations.
We apply the $k$-step BDF convolution quadrature to discretize the time-fractional derivative with order $\alpha\in (0,1)$,
and modify the starting steps in order to achieve optimal convergence rate.
This method has already been well-studied for the linear fractional evolution equations in Jin, Li and Zhou \cite{JinLiZhou:correction},
while the numerical analysis for the nonlinear problem is still missing in the literature.
By splitting the nonlinear potential term into an irregular linear part
and a smoother nonlinear part, and using the generating function technique,
we prove that the convergence order of the corrected BDF$k$ scheme is $O(\tau^{\min(k,1+2\alpha-\epsilon)})$,
without imposing further assumption on the regularity of the solution.
Numerical examples are provided to support our theoretical results.
\\


\end{abstract}

\setlength\abovedisplayskip{3.5pt}
\setlength\belowdisplayskip{3.5pt}

\section{\bf Introduction}\label{Se:intr}
Fractional partial differential equations (PDEs) have been drawing increasing attention over
the past several decades, due to their capability to describe anomalous diffusion processes,
in which the mean square variance of particle displacements grow sublinearly/superlinear with the time,
instead of the linear growth for a Gaussian process.
Nowadays those models have been successfully employed in many practical applications,
including dynamics of single-molecular protein \cite{kou2008stochastic},
flow in highly heterogeneous aquifer\cite{berkowitz2002physical} and
thermal diffusion in fractal domains\cite{nigmatullin1986realization}, to name but a few; see \cite{Metzler:2014} for an extensive list.

The aim of this paper is to study high-order time stepping schemes for solving the initial-boundary
value problem for the semilinear subdiffusion equation:
\begin{equation}\label{Eqn:fde}
\left\{\begin{aligned}
\dalt u-\Delta u&=f(u)&&\text{ in } \Omega\times (0,T),\\
u &=0 &&\text{ on  }  \pt\Omega\times (0,T),\\
u(0)&=u_0 &&\text{ in }\Omega ,
\end{aligned}\right.
\end{equation}
where $\Omega$ denotes  a  bounded, convex
domain in $\R^d$ with smooth boundary,  and $\Delta$ denotes the Laplacian on $\Omega$
with a homogenous Dirichlet boundary condition. Here $\dalt u$ denotes the left-sided Caputo
fractional derivative of order $\al\in(0,1)$ with respect to
$t$ and it is defined by \cite[pp.\,91]{KilbasSrivastavaTrujillo:2006}
\begin{equation*}\label{McT}
\partial_t^\alpha u(t):= \frac{1}{\Gamma(1-\alpha)} \int_0^t(t-s)^{-\alpha}u'(s)\, \d s ,\quad \mbox{with}\quad \Gamma(z):=\int_0^\infty s^{z-1}e^{-s}\d s.
\end{equation*}

Throughout the paper, we assume that the initial data $u_0$ is smooth and compatible with the homogeneous Dirichlet boundary condition, 
and $f:\mathbb{R}\rightarrow \mathbb{R}$ is a globally smooth function, e.g.,
$f \in C^3(\mathbb{R})$. Moreover,
we assume that the nonlinear subdiffusion problem \eqref{Eqn:fde} has a unique global solution $u\in C([0,T]\times\bar\Omega)$.
One typical example is the time-fractional Allen-Cahn equation, i.e., $f(u)=u-u^3$,
whose well-posedness and smoothing properties have already been investigated in \cite{DuYangZhou:AllenCahn}.

High-order time stepping schemes for solving the linear time-fractional evolution problems
have been intensively studied in recent years;
see \cite{JinLazarovZhou:overview} (and the references therein) for a concise overview.
Roughly speaking, there are two prominent types of schemes: piecewise polynomial
interpolation (e.g., \cite{alikhanov2015new, GaoSunZhang:2014, lin2007finite, Sun:2006}) and convolution quadrature (CQ)
(e.g., \cite{CuestaLubichPalencia:2006, Jin:SISC2016, YusteAcedo:2005, Zeng:2013}).
To the first group belongs the popular method using a piecewise linear interpolation (also known
as L1 scheme). Lin and Xu \cite{lin2007finite} developed the scheme for fractional diffusion, and analyzed the
stability and convergence rate; see also \cite{Sun:2006}. The discretization has a local truncation error $O(\tau^{2-\alpha})$ where $\tau$ denotes the step size in time,
provided that the solution is smooth enough in time.  The argument could be extended to high-order methods
using piecewise polynomial interpolation \cite{alikhanov2015new, GaoSunZhang:2014}.
In the second group, CQ developed by Lubich \cite{Lubich:1986, Lubich:1988} provided a systematic framework
to construct high-order numerical schemes, and has been the
foundation of many early works. Due to its particular construction, it naturally inherits the stability and accuracy of standard linear multistep methods,
which greatly facilitates the analysis of resulting numerical schemes.
However, for both techniques with uniform meshes, the desired convergence rates can be obtained only
if data is sufficiently smooth and compatible, which is generally not valid.
Otherwise, most of popular schemes can only achieve a first-order accuracy \cite{JinLazarovZhou:L1, JinLiZhou:correction}.
For the linear problem, the desired high-order convergence rates can be restored by correcting the first several time steps
\cite{Jin:SISC2016, JinLiZhou:correction, Yan:2018L1}, even for nonsmooth problem data.
See also \cite{Stynes:2017, LiaoLiZhang:2018} for the application of {L1} scheme with graded meshes,
\cite{McLeanMustapha:2015, Mustapha:2014, MustaphaMcLean:2013} for an analysis of discontinuous
Galerkin method 
and \cite{ChenXuHesthaven:2015, LiXu:2009, Zayernouri:2015uni} for studies of spectral methods.

However, there is fewer work on nonlinear subdiffusion problems. The first rigorous analysis was given in \cite{JinLiZhou:nonlinear},
where Jin et al. proposed a general framework for mathematical and numerical analysis of the nonlinear equation \eqref{Eqn:fde}
with a globally Lipschitz continuous potential term $f(u)$. A time stepping scheme based on backward Euler CQ scheme or L1 method was studied and
a uniform-in-time convergence rate $O(\tau^\alpha)$ was proved. Then it was proved in \cite{Karaa:nonlinear}
that  the convergence rate of the backward Euler CQ scheme is $O(\tau)$ at a fixed time even for the nonsmooth data.
 As far as we know, there is no theoretical study on high-order schemes for the nonlinear problem \eqref{Eqn:fde} based on confirmed solution regularity.
Therefore, in this paper, we aim to study high-order time stepping schemes based on CQ generated by $k$-step BDF method.
This work is motivated by our preceding studies on the corrected BDF$k$ schemes for linear subdiffusion equations \cite{Jin:SISC2016, JinLiZhou:correction}.

To discretize the fractional derivative, we let $0=t_0<t_1<\ldots<t_N=T$ be a uniform partition of the time interval $[0,T]$, with grid
points $t_n=n\tau$ and step size $\tau=T/N$. Upon rewriting the Caputo derivative
$\partial_t^\alpha u$ as a Riemann-Liouville one \cite[pp. 91]{KilbasSrivastavaTrujillo:2006}, we consider the following fully implicit time stepping scheme:
for the given initial value $u_0$, find $u_n$, $n=1,2,\ldots, N$, such that
\begin{align}\label{TD-scheme}
\begin{aligned}
&\bar\partial_\tau^\alpha (u_n-u_0)-\Delta u_n= f(u_n) ,
\end{aligned}
\end{align}
where $u_n,\,n=1,2,\cdots,N$ are the approximations to the exact solutions $u(t_n)$, and $\bdalt\varphi_n$ denotes the convolution quadrature generated by $k$-step BDF, $k=1,2,\cdots,6$
with the definition %
\begin{equation}\label{Eq:BDF-CQ}
  \bdalt\varphi_n:=\frac{1}{\tau^\al}\sum_{i=0}^{n} {\omega_i^{(\alpha)}}\varphi_{n-i}.
\end{equation}
The coefficients $\{\omega_i^{(\alpha)}\}_{i=0}^{\infty}$ can be computed either by the fast
Fourier transform\cite{podlubny1998fractional,sousa2012approximate} or recursion\cite{wu2014determination} in the following series expansion
\begin{equation}\label{eqn:gen-k}
  \delta_\tau(\xi)^\al=\frac{1}{\tau^\al} \sum_{i=0}^{\infty} {\omega_i^{(\alpha)}}\xi^i\quad\text{with}\quad \delta_\tau(\xi)=\frac{1}{\tau} \sum_{i=1}^{k} \frac{1}{i}(1-\xi)^i.
\end{equation}

For linear subdiffusion problem, it has been shown in \cite{JinLiZhou:correction} that the scheme \eqref{TD-scheme} is only first-order accurate in general.
However, the optimal  order $O(\tau^k)$ of the BDF$k$ scheme could be restored by correcting the first $k-1$ steps.
For example, we split the source term $f$ into $f(t)=f(0) + (f(t)-f(0))$ and approximate $f(0)$ by
$\bar\partial_\tau \partial_t^{-1}f(0)$, with a similar treatment of the initial data. This leads to
a simple modification at the first step and restores the $O(\tau^2)$ accuracy for any fixed $t_n>0$ \cite{LubichSloanThomee:1996, CuestaLubichPalencia:2006, Jin:SISC2016}.
This motivates us to decompose the nonlinear potential term $f(u)$ by
\begin{equation}\label{eqn:taylor}
 f(u(t))= f(u_0) + f'(u_0) (u(t)-u_0) + R(u(t);u_0).
 \end{equation}
Then the residue part, $R(u(t);u_0)=O((u(t)-u_0)^2)$, is more regular in the time direction.
As a result, the semilinear equation can be reformulated by
\begin{equation}\label{eqn:fde-m}
\begin{aligned}
\dalt u(t)-(\Delta + f'(u_0)I) u(t)&= f(u_0) - f'(u_0)u_0 + R(u(t);u_0),
\end{aligned}
\end{equation}
where $I$ denotes the identity operator. Therefore, by letting
\begin{equation}\label{eqn:A}
g_0 = f(u_0) - f'(u_0)u_0\quad\text{and} \quad A=\Delta + f'(u_0)I,
\end{equation}
we can modify the BDF$k$ scheme \eqref{TD-scheme} by
\begin{equation}\label{eqn:BDF-CQ-m}
\left\{\begin{aligned}
&\bdalt (u-u_0)_n-Au_n=g_0+a_n^{(k)}(Au_0+g_0)+R(u_n;u_0),\quad &&1\leq n\leq k-1,\\
&\bdalt (u-u_0)_n-Au_n=g_0+R(u_n;u_0),\quad &&k\leq n\leq N,
\end{aligned}\right.
\end{equation}
where the unknown coefficients $a_n^{(k)}$ were given in \cite[Table 1]{JinLiZhou:correction}.
\begin{table}[h!]
  \centering
    \caption{The coefficients $a_n^{(k)}$}\label{Tab:ank}
    \vskip-7pt
  \begin{tabular}{|c|cc ccc|}%
  \hline
   BDF$k$ & $a_1^{(k)}$ & $a_2^{(k)}$& $a_3^{(k)}$ & $a_4^{(k)}$ & $a_5^{(k)}$\\[0.5ex]\hline%
  $k=2$ & $\frac{1}{2}$   & &  &  &\\[0.5ex]\hline%
  $k=3$ & $\frac{11}{12}$ & $-\frac{5}{12}$ & & &\\[0.5ex]\hline%
  $k=4$ & $\frac{31}{24}$ & $-\frac{7}{6}$ & $\frac{3}{8}$ &  &  \\[0.5ex]\hline
  $k=5$ & $\frac{1181}{720}$ & $-\frac{177}{80}$ & $\frac{341}{240}$ & $-\frac{251}{720}$ &  \\[0.5ex]\hline
  $k=6$ & $\frac{2837}{1440}$ & $-\frac{2543}{720}$ & $\frac{17}{5}$ & $-\frac{1201}{720}$ & $\frac{95}{288}$ \\[0.5ex]
  \hline
\end{tabular}
\end{table}

By rearranging terms, the modified BDF$k$ scheme \eqref{eqn:BDF-CQ-m} is equivalent to
\begin{equation}\label{eqn:BDF-CQ-m-r}
\left\{\begin{aligned}
&\bdalt (u-u_0)_n-\Delta u_n= a_n^{(k)}(\Delta u_0 + f(u_0)) + f(u_n),\quad &&1\leq n\leq k-1,\\
&\bdalt (u-u_0)_n-\Delta u_n=f(u_n),\quad &&k\leq n\leq N,
\end{aligned}\right.
\end{equation}
which is consistent to the BDF$k$ scheme for the linear subdiffusion problem \cite{Jin:SISC2016, JinLiZhou:correction}.

The main result of this paper is to derive an error estimate in $L^\infty(\Omega)$ for the novel time stepping scheme \eqref{eqn:BDF-CQ-m}.
In particular, if $u_0\in \{u\in C(\bar\Omega),~u=0~\text{on}~\partial\Omega,~\text{and}~ \Delta u \in C(\bar\Omega) \}$, we prove that
(see Theorem \ref{thm:error})
\begin{equation}\label{eqn:err-0}
\| u_n - u(t_n) \|_{L^\infty(\Omega)} \le c_T  t_n^{\alpha - \min(k,1+2\alpha-\ep)} \tau^{\min(k,1+2\alpha-\ep)}.
\end{equation}
This estimate is interesting, because the source term $f(u)\in {W^{1+\alpha-\ep,1}(0,T;L^\infty\II)}$ in general,
which is nonsmooth in the time direction, and intuitively one only expects the convergence order $O(\tau^{\min(k,1+\alpha-\epsilon)})$ \cite[Table 8]{JinLiZhou:correction}.
However, the estimate \eqref{eqn:err-0} indicates that the best convergence rate of the BDF$k$ scheme is almost $O(\tau^{1+2\alpha})$.
The restriction of the convergence order comes from the low regularity of the remainder $R(u;u_0)$,
even though the initial data $u_0$ is smooth and compatible with boundary condition.
This phenomena contrasts sharply with its normal parabolic counterpart, i.e., $\alpha=1$.
For example, in \cite{CrouzeixThomee}, it has been proved that the time stepping schemes of the semilinear parabolic equation
fail to achieve the best convergence rate only if the initial data is not regular enough.

The rest of the paper is organized as follows. In section \ref{sec:prelim}, we provide some preliminary results
about the solution regularity which will be intensively used in error estimation. The error analysis of the time stepping scheme \eqref{eqn:BDF-CQ-m}
is established in section \ref{sec:error}.
Then the fully discrete scheme are analyzed in section \ref{sec:fully}.
Finally, in section \ref{sec:numerics}, we present some numerical results which support and illustrate our theoretical findings.
Throughout this paper, the notation $c$ denotes a generic constant, which may vary at different occurrences, but it
is always independent of the time step size $\tau$ and spatial mesh size $h$.

\section{Preliminary results}\label{sec:prelim}
In this section, we shall present some regularity results which will be actively used in the next section.
As we introduced, we always assume that the semilinear subdiffion problem \eqref{Eqn:fde}
has a unique global solution $u\in C([0,T]\times\bar \Omega)$ (e.g., the time-fractional Allen-Cahn equation \cite{DuYangZhou:AllenCahn}).

\subsection{{Solution representation}}
First, we introduce a representation of the solution to problem \eqref{Eqn:fde} by Laplace transform. For simplicity, we let
$g(t):=f(u(t))$ and $w(t):=u(t)-u_0$. Then it is easy to see that the function $w(t)$ satisfies the equation
\begin{equation*}
\dalt w(t)-\Delta w(t)=\Delta u_0+g(t)
\end{equation*}
with initial condition $w(0)=0$. Taking Laplace transform, denoted by $\hat{}$ , we have
\begin{equation*}
  z^\al{\hat{w}}(z)-\Delta {\hat{w}}(z)=z^{-1}\Delta u_0+\hat{g}(z),
\end{equation*}
which implies that ${\hat{w}}(z)=(z^\al-\Delta )^{-1}(z^{-1}\Delta u_0+\hat{g}(z))$. With inverse Laplace transform and convolution rule,
the solution $u(t)$ can be explicitly expressed by
\begin{equation}\label{eqn:sol-rep}
  u(t)=(I + F(t)\Delta) u_0+\int_{0}^{t} E(t-s)f(u(s))\d s,
\end{equation}
where the operators $F(t)$ and $E(t)$ are defined by
\begin{equation}\label{eqn:op}
  F(t)=\frac{1}{2\pi\i}\int_{\Gamma_{\theta,\delta}} e^{zt}z^{-1}(z^\al-\Delta )^{-1}\d z\quad\text{and}\quad
  E(t)=\frac{1}{2\pi\i}\int_{\Gamma_{\theta,\delta}} e^{zt}(z^\al-\Delta )^{-1}\d z,
\end{equation}
respectively, where $\Gamma_{\theta,\delta}$ denotes the integral contour
\begin{equation*}
  \Gamma_{\theta,\delta}=\{z\in\mathbb{C}:|z|=\delta,|\arg z|\leq \theta\} \cup \{z\in\mathbb{C}: z=\rho e^{\pm\i\theta},\rho\geq\delta\},
\end{equation*}
oriented with an increasing imaginary part with a fixed angle $\theta\in(\pi/2,\pi)$. 

In this paper, we shall derive some estimates in $L^\infty(\Omega)$ norm, which requires the resolvent estimate.
Let us consider the second-order partial differential operator
\begin{equation*}
 L u = -\Delta u + q u,
\end{equation*}
with the homogeneous Dirichlet boundary condition.
Here we assume that $q\in L^\infty(\Omega)$ and $q(x)\ge 0$ for all $x\in \Omega$. This implies that $L$ is positively definite, i.e.,
\begin{equation*}
 (L u , u) \ge c \| \nabla u \|_{L^2(\Omega)}^2,\qquad \text{for all}~~u\in H_0^1(\Omega).
\end{equation*}
Then the following resolvent estimate holds: for any angle $\phi\in(\pi/2,\pi)$ (see \cite[Theorem 1.1]{Bakaev}, \cite[Theorem 2.1]{Bakaev:2003} or \cite[Theorem 1]{Stewart})
\begin{equation}\label{eqn:resol}
  \| (z + L)^{-1}  \|_{C(\bar\Omega)\rightarrow C(\bar\Omega)} \le c  |z|^{-1}   \qquad \text{for}~z\in \Sigma_{\phi}= \{ z\in\mathbb{C}\backslash \{ 0 \}: \text{arg}(z) \in (-\phi, \phi)  \}.
\end{equation}

From now on, we assume that the initial condition $u_0$ is smooth enough and
 compatible to the homogenous Dirichlet boundary condition, i.e.,
\begin{equation}\label{eqn:ini-D}
u_0 \in D=\{u\in C(\bar\Omega),~u=0~\text{on}~\partial\Omega,~\text{and}~ \Delta u \in C(\bar\Omega) \}.
\end{equation}
Then by the resolvent estimate \eqref{eqn:resol}, it is easy to observe that the operators $F$ and $E$, defined in \eqref{eqn:op}, satisfy
the following regularity estimate that for  $\ell=0,1,2,\ldots,$
\begin{equation}\label{eqn:smoothing}
 t \| \partial_t^{(\ell)}E(t) v  \|_{L^\infty(\Omega)}  + \| \partial_t^{(\ell)}F(t) v  \|_{L^\infty(\Omega)} \le c t^{\alpha-\ell} \| v \|_{L^\infty(\Omega)}\qquad \forall ~~v\in C(\bar\Omega).
\end{equation}
The estimates with $L^2(\Omega)$-norm have already been confirmed in \cite[Lemma 3.4]{JinLiZhou:nonlinear}
by using resolvent estimate.
The proof of \eqref{eqn:smoothing} is similar to that, and hence is omitted here.

\subsection{{Solution regularity}}

With the help of \eqref{eqn:smoothing}, we are ready to state the following lemma on the regularity of the solution to the nonlinear
subdiffusion equation \eqref{Eqn:fde}.

\begin{theorem}\label{thm:reg-u}
We assume that $u_0\in D$ with the space $D$ defined by \eqref{eqn:ini-D}. Besides, suppose that the  problem \eqref{Eqn:fde}
has a unique global solution $u\in C([0,T]\times\Omega)$. Then
$u\in C^{\alpha}([0,T];C(\bar\Omega))\cap C^\ell((0,T];C(\bar\Omega))$, $\ell=1,2,3$, 
and it satisfies the a priori estimate
\begin{equation}\label{Eq:inf}
  \| \partial_t^\ell u(t) \|_{L^\infty(\Omega)} \le c t^{\alpha-\ell} ,\quad \text{for}~\ell=1,2,3,
\end{equation}
where the constant $c$ depends on $\alpha, T$ and $u_0$.
\end{theorem}
\begin{proof}
The H\"older continuity $u\in C^{\alpha}([0,T];C(\bar\Omega))$ and the estimate \eqref{Eq:inf} with $\ell=1$ are direct results of the solution representation \eqref{eqn:sol-rep},
the estimate \eqref{eqn:smoothing}, and the Banach fixed point theorem. The argument is identical
to the proof of \cite[Theorem 3.1]{JinLiZhou:nonlinear}, and hence omitted here. 

Now we turn to the estimate \eqref{Eq:inf} with $\ell=2$, which requires more discussion.
First, we take derivative on the solution representation \eqref{eqn:sol-rep} and obtain
\begin{equation}\label{eqn:der1}
\begin{aligned}
u'(t)&= \frac{d}{dt} F(t) \Delta u_0 + \frac{d}{dt} \int_0^t E(s) f(u(t-s))\,\d s \\
&= E(t) [\Delta u_0 + f(u_0)] +  \int_0^t E(s) f'(u(t-s)) u'(t-s)\,\d s, \\
\end{aligned}
\end{equation}
where we use the fact that $F'(t) = E(t)$. Here we note that both $E(t)$ and $u'(t)$ are weakly singular near $t=0$.
Therefore, we multiply $t^{2-\alpha}$ on \eqref{eqn:der1} to compensate for the singularity before differentiation.
Then
\begin{equation}\label{Eq:tuprime}
\begin{aligned}
  t^{2-\al}u'(t)&=t^{2-\al}E(t)[\Delta u_0+f(u_0)]+t^{2-\al} \int_0^t E(s) f'(u(t-s)) u'(t-s)\,\d s \\
  &=t^{2-\al}E(t)(\Delta u_0+f(u_0))+t^{1-\al}\int_{0}^{t} (t-s)E(t-s)f'(u(s))u'(s)\d s\\
  &\quad +t^{1-\al}\int_{0}^{t} E(s)f'(u(t-s))(t-s)u'(t-s)\d s\\
  & =: \sum_{i=1}^3 I_i(t).
\end{aligned}\end{equation}
After taking derivative of the first term $I_1$,
we apply the  estimate \eqref{eqn:smoothing} to obtain that
\begin{equation*}\label{Eq:I1}
\begin{aligned}
\|\partial_tI_1(t)\|_{L^\infty(\Omega)}&=\Big\|\Big((2-\al)t^{1-\al}E(t)+t^{2-\al}E'(t)\Big)[\Delta u_0+f(u_0)]\Big\|_{L^\infty(\Omega)}\\
&\leq ct^{1-\al}\|E(t)[\Delta u_0+f(u_0)]\|_{L^\infty(\Omega)}+t^{2-\al}\|E'(t)[\Delta u_0+f(u_0)]\|_{L^\infty(\Omega)}\\
&\le c\| \Delta u_0+f(u_0) \|_{L^\infty(\Omega)},
\end{aligned}
\end{equation*}
where we use the fact that $\Delta u_0+f(u_0) \in C(\bar\Omega)$.
The derivative of the second term $I_2$ in \eqref{Eq:tuprime} can be estimate analogously. Using the estimate \eqref{eqn:smoothing}, we have
$$ \lim_{t\rightarrow 0} \| t E(t) \|_{ C(\bar \Omega) \rightarrow C(\bar \Omega)} = 0, $$
which together with the triangle's inequality and \eqref{Eq:inf} for $\ell=1$ implies that
\begin{equation*}
\begin{aligned}
\|\partial_tI_2(t)\|_{L^\infty(\Omega)}&\le c t^{-\alpha}\int_{0}^{t} (t-s) \|E(t-s)f'(u(s))u'(s)\|_{L^\infty(\Omega)} \d s\\
&\quad + ct^{1-\alpha} \int_{0}^{t} \|E(t-s)f'(u(s))u'(s)\|_{L^\infty(\Omega)} \d s \\
&\quad + ct^{1-\alpha} \int_{0}^{t} (t-s) \|E'(t-s)f'(u(s))u'(s)\|_{L^\infty(\Omega)} \d s \\
&\le c t^{-\alpha}\int_{0}^{t} (t-s)^\alpha s^{\alpha-1} \d s
+ ct^{1-\alpha} \int_{0}^{t}(t-s)^{\alpha-1} s^{\alpha-1} \d s \le c_T.
\end{aligned}
\end{equation*}
Similarly, the derivative of {the} third term can be bounded by
\begin{equation*}
\begin{aligned}
\|\partial_tI_3(t)\|_{L^\infty(\Omega)}&\le ct^{-\al}\int_{0}^{t} s\|E(t-s)f'(u(s))u'(s)\|_{L^\infty(\Omega)} \d s\\
&\quad + ct^{1-\alpha} \int_{0}^{t}  \|E(t-s)[f'(u(s))u'(s)+sf''(u(s))(u'(s))^2]\|_{L^\infty(\Omega)} \d s \\
&\quad + ct^{1-\alpha} \int_{0}^{t}  s \|E(t-s) f'(u(s))u''(s)\|_{L^\infty(\Omega)} \d s \\
&\le c_T +  ct^{1-\alpha} \int_{0}^{t}  (t-s)^{\alpha-1}s \| u''(s)\|_{L^\infty(\Omega)} \d s.
\end{aligned}
\end{equation*}
As a result, we achieve at
$$\|\partial_t[t^{2-\alpha}u'(t)]\|_{L^\infty(\Omega)} \le c + ct^{1-\alpha} \int_{0}^{t}  (t-s)^{\alpha-1}s \| u''(s)\|_{L^\infty(\Omega)} \d s. $$
Then we apply  \eqref{Eq:inf} for $\ell=1$ again and use the triangle inequality to obtain that
\begin{equation}\label{eqn:est1}
\begin{aligned}
t^{2-\alpha}\|u''(t)\|_{L^\infty(\Omega)} &\le c + c t^{1-\alpha} \int_{0}^{t}  (t-s)^{\alpha-1}s \| u''(s)\|_{L^\infty(\Omega)} \d s.
\end{aligned}
\end{equation}
In order to derive a uniform bound of $t^{2-\alpha}\|u''(t)\|_{L^\infty(\Omega)}$, we
multiply $e^{-\sigma t}$ on the inequality \eqref{eqn:est1} for some parameter $\sigma>0$ to be determined,
and obtain that
\begin{equation}\label{eqn:gr-01}
\begin{aligned}
&\quad e^{-\sigma t}t^{2-\alpha}\|u''(t)\|_{L^\infty(\Omega)}\\
&\le  ce^{-\sigma t} + ce^{-\sigma t} \int_{0}^{t}  t^{1-\alpha} (t-s)^{\alpha-1}s \| u''(s)\|_{L^\infty(\Omega)} \d s\\
&\le ce^{-\sigma t}   + c \Big[\max_{t\in[0,T]}e^{-\sigma t}t^{2-\alpha}\| u''(t)\|_{L^\infty(\Omega)} \Big] \int_{0}^{t} t^{1-\alpha} (t-s)^{\alpha-1}e^{-\sigma(t-s)} s^{\alpha-1} \d s\\
&\le  ce^{-\sigma t}   + c \big(T/\sigma\big)^{\frac\alpha2}\max_{t\in[0,T]}e^{-\sigma t}t^{2-\alpha}\| u''(t)\|_{L^\infty(\Omega)},
\end{aligned}
\end{equation}
where we use the estimate that
\begin{equation}\label{eqn:gr-02}
\begin{aligned}
\int_{0}^{t} t^{1-\alpha} (t-s)^{\alpha-1}e^{-\sigma (t-s)} s^{\alpha-1} \d s
&= t^{\alpha}\int_{0}^{1} e^{-\sigma t s} s^{\alpha-1} (1-s)^{\alpha-1} \d s\\
&= \big(t/\sigma\big)^{\frac\alpha2} \int_{0}^{1} [e^{-\sigma t s}(\sigma ts)^{\frac\alpha2}] s^{\frac\alpha2-1} (1-s)^{\alpha-1} \d s\\
&\le  c \big(t/\sigma\big)^{\frac\alpha2}\int_{0}^{1} s^{\frac\alpha2-1} (1-s)^{\alpha-1} \d s \le c \big(T/\sigma\big)^{\frac\alpha2}.
\end{aligned}
\end{equation}
Finally, by choosing a sufficient large $\lambda$ such that $2c \big(T/\sigma\big)^{\frac\alpha2}< 1$, we obtain that
\begin{equation}\label{eqn:gr-03}
\max_{s\in[0,T]} e^{-\lambda s}s^{2-\al}\|u''(s)\|_{L^\infty(\Omega)}\leq c,
\end{equation}
which confirms the assertion \eqref{Eq:inf} with  $\ell=2$.

Now we turn to the case $\ell=3$ and give a brief proof. The basic idea of this argument is identical to that of $\ell=2$.
With the definition of $I_i$  in \eqref{Eq:tuprime} and  the estimate \eqref{eqn:smoothing} of the solution operator $E(t)$, we have the bound that
\begin{equation*}\label{Eq:I1-2}
\begin{aligned}
\|\partial_{tt}I_1(t)\|_{L^\infty(\Omega)}
&\le c \sum_{k=0}^2 t^{k-\al}\Big\|\frac{\d^k}{\d t^k}E(t)[\Delta u_0+f(u_0)]\Big\|_{L^\infty(\Omega)}
\le c t^{-1}.
\end{aligned}
\end{equation*}
For the second term, we use the splitting
\begin{equation*} I_2 = t^{-\alpha}\int_0^t (t-s)^2 E(t-s)f'(u(s))u'(s) \, \d s + t^{-\alpha}\int_0^t (t-s) E(t-s) s f'(u(s))u'(s)\, \d s,\end{equation*}
and the fact that
\begin{equation*}
\lim_{t\rightarrow 0} \| t E(t) \|_{ C(\bar \Omega) \rightarrow C(\bar \Omega)} + \| tu'(t) \|_{L^\infty(\Omega)} = 0,
\end{equation*}
and hence derive that
\begin{equation*}
\begin{aligned}
\|\partial_{tt}I_2(t)\|_{L^\infty(\Omega)}&\le c \sum_{k=0}^2 t^{-(2-k)-\al} \int_0^t \sum_{m=0}^k (t-s)^{2+m-k} \big|\big|\frac{\d^{m}}{\d t^{m}}E(t-s)[f'(u(s))u'(s)]\big|\big|_{L^\infty(\Omega)} \,\d s\\
 &\quad + c \sum_{k=0}^1 t^{-(2-k)-\al} \int_0^t   \big|\big|\frac{\d^{k}}{\d t^{k}}[(t-s)E(t-s)] [s(f'(u(s))u'(s)]\big|\big|_{L^\infty(\Omega)} \,\d s\\
 &\quad + c t^{-2-\al} \int_0^t   \big|\big|\frac{\d}{\d t}[(t-s)E(t-s)] [\partial_s(s(f'(u(s))u'(s))]\big|\big|_{L^\infty(\Omega)} \,\d s.
\end{aligned}
\end{equation*}
Then the estimates \eqref{eqn:smoothing} and  \eqref{Eq:inf} with  $\ell=1,2$ imply that
\begin{equation*}\label{Eq:I2-2}
\begin{aligned}
\|\partial_{tt}I_2(t)\|_{L^\infty(\Omega)}&\le c \sum_{k=0}^2 t^{-(2-k)-\al} \int_0^t (t-s)^{\alpha-k+1}s^{\alpha-1}\,\d s + c \sum_{k=0}^1 t^{-(2-k)-\al} \int_0^t  (t-s)^{\alpha-k} s^\alpha \,\d s\\
 &\quad + c t^{-2-\alpha}\int_0^t  (t-s)^{\alpha-1} (s^{\alpha-1} +s^{2\alpha-1})\,\d s
\le c t^{\alpha-1} .
\end{aligned}
\end{equation*}
The same argument also works for the third term $I_3$ in \eqref{Eq:tuprime}:
\begin{equation*}
\begin{aligned}
\|\partial_{tt}I_3(t)\|_{L^\infty(\Omega)}&\le c \sum_{k=0}^2 t^{-(2-k)-\al} \int_0^t \sum_{m=0}^k s^{2+m-k} \big|\big|E(t-s)[\partial_s^m(f'(u(s))u'(s))]\big|\big|_{L^\infty(\Omega)} \,\d s\\
 &\quad + c \sum_{k=0}^1 t^{-(2-k)-\al} \int_0^t   \big|\big| (t-s)E(t-s)  \partial_s^{m}[s(f'(u(s))u'(s)]\big|\big|_{L^\infty(\Omega)} \,\d s\\
 &\quad + c t^{-2-\al} \int_0^t   \big|\big|\frac{d}{dt}[(t-s)E(t-s)] [\partial_s(s(f'(u(s))u'(s))]\big|\big|_{L^\infty(\Omega)} \,\d s\\
 &\le  c t^{\alpha-1} + t^{-\alpha}\int_0^t (t-s)^{\alpha-1}s^2 \| u^{(3)}(s) \|_{L^\infty(\Omega)}  \,\d s.
 \end{aligned}
\end{equation*}
As a result, we conclude that
$$\|\partial_{tt}[t^{2-\alpha}u'(t)]\|_{L^\infty(\Omega)} \le c t^{-1} + ct^{-\alpha} \int_{0}^{t}  (t-s)^{\alpha-1}s^2 \| u''(s)\|_{L^\infty(\Omega)} \d s. $$
Then we apply the estimate \eqref{Eq:inf} for $\ell=1,2$ and obtain that
$$t^{3-\alpha}\| \partial_t^3 u (t)]\|_{L^\infty(\Omega)} \le c + ct^{1-\alpha} \int_{0}^{t}  (t-s)^{\alpha-1}s^2 \| u''(s)\|_{L^\infty(\Omega)} \d s. $$
Finally, the arguments in \eqref{eqn:gr-01}-\eqref{eqn:gr-03} yield the desired assertion \eqref{Eq:inf} for $\ell=3$.
\end{proof}\vskip5pt

\begin{remark}\label{rem:CD}
Under the condition that $u_0\in D$, it is not always valid that $(I + F(t)\Delta) u_0 \in C([0,T];D)$.
This is because $\Delta (I + F(t)\Delta) u_0$ is compatible with the homogeneous Dirichlet boundary condition for all $t>0$, while
the initial condition $\Delta u_0\in C(\bar \Omega)$. As a result, $\Delta (I + F(t)\Delta) u_0$ is not continuous to the initial condition with
 $L^\infty(\Omega)$ norm. 
However, if
$$u_0\in D = \{u,\Delta u\in C(\bar\Omega)~\text{and}~u=\Delta u=0~\text{on}~\partial\Omega \},$$
then we can conclude that $(I + F(t)\Delta)u_0\in C([0,T];D)$.
\end{remark}\vskip5pt

\subsection{Regularity of remainder $R(u(t);u_0)$}

Recall the expansion of the nonlinear term in \eqref{eqn:taylor}. The regularity of the
remainder part $R(u;u_0)$ plays an important role in the error analysis.
This motivates us to derive  regularity results of $R(u;u_0)$ in the Bochner-Sobolev spaces. For any $s\ge 0$ and
$1\le p < \infty$, we denote by $W^{s,p}(0,T;B)$ the space of functions $v:(0,T)\rightarrow B$,
with the norm defined by interpolation, where $B$ denotes a Banach space. Equivalently, the space is equipped with the quotient norm
\begin{equation*}\label{quotient-norm}
\|v\|_{W^{s,p}(0,T;B)}:=
\inf_{\widetilde v}\|\widetilde v\|_{W^{s,p}({\mathbb R};B)} ,
\end{equation*}
where the infimum is taken over all possible extensions $\widetilde v$ that extend $v$ from $(0,T)$ to ${\mathbb R}$.
For any $0<s< 1$, the Sobolev--Slobodeckij seminorm $|\cdot|_{W^{s,p}(0,T;B)}$ is defined by
\begin{equation}\label{eqn:SS-seminorm}
   | v  |_{W^{s,p}(0,T;B)}^p := \int_0^T\hskip-5pt\int_0^T \frac{\|v(t)-v(\xi)\|_{B}^p}{|t-\xi|^{1+ps}} \,\d t \, \d\xi ,
\end{equation}
and the full norm $\|\cdot\|_{W^{k+s,p}(0,T;B)}$, with $k\ge 0$ and $k\in \mathbb{N}$, is defined by
\begin{equation*}
\|v\|_{W^{k+s,p}(0,T;B)}^p = \sum_{m=0}^k\|\partial_t^m v \|_{L^p(0,T;B)}^p+|\partial_t^k v |_{W^{s,p}(0,T;B)}^p .
\end{equation*}
%
%
Then the regularity of $R(u;u_0)$ is shown in the following theorem.

\begin{theorem}\label{thm:reg-R}
Suppose that the assumptions in Theorem \ref{thm:reg-u} hold. Then the remainder part $R(u;u_0)$, which is defined by \eqref{eqn:taylor},
has the regularity
\begin{equation*}
  R(u(t);u_0)\in W^{1+2\al-\epsilon,1}(0,T;C(\bar\Omega))\cap C^3((0,T];C(\bar\Omega))
\end{equation*}
for any arbitrary small $\epsilon>0$.
\end{theorem}
\begin{proof}
By the definition of $R(u(t);u_0)$ and the integral form of the remainder in Taylor's expansion,
we may rewrite $R(u(t);u_0)$ as
\begin{equation*}\label{Eq:reRu}
  R(u(x,t);u_0)=\int_{u_0(x)}^{u(x,t)} (u(x,t)-\xi)f''(\xi) \, \d \xi.
\end{equation*}
It is easy to observe that
\begin{equation*}
\begin{aligned}
  \partial_t^3 R(u(x,t);u_0)& =  \partial_t^2 \Big(u'(x,t)\int_{u_0(x)}^{u(x,t)}  f''(\xi)\, \d \xi\Big)\\
  & = u'''(x,t) \int_{u_0(x)}^{u(x,t)} f''(\xi)\, \d\xi +3u'(x,t)u''(x,t)f''(u(x,t))\\
  &\quad +(u'(x,t))^3f'''(u(x,t)).
\end{aligned}
\end{equation*}
Then using the facts that $f$ is smooth and $u\in C^3((0,T];C(\bar\Omega))$ by Theorem \ref{thm:reg-u}, we conclude that
$ R(u;u_0) \in C^3((0,T];C(\bar\Omega))$. Therefore, it suffices to show
that $ R(u(t);u_0)\in W^{1+2\al-\epsilon,1}(0,T;C(\bar\Omega))$.
To this end, we shall confirm this claim by investigating the following two cases.

\textbf{Case 1. $\alpha\in(0,1/2)$.} Obviously,
we have $R(u ;u_0)\in C([0,T]\times\bar \Omega)$. Define
\begin{equation*}\label{Eq:DtRu}
  w(x,t)=\partial_t R(u(x,t);u_0)=\int_{u_0(x)}^{u(x,t)} u'(x,t)f''(\xi)\, \d \xi.
\end{equation*}
Then we observe that
\begin{equation*}
\begin{aligned}
\|w(t)\|_{L^\infty(\Omega)}
 &\leq c\|u(t)-u_0\|_{L^\infty(\Omega)}\|u'(t)\|_{L^\infty(\Omega)} \max_{t\in[0,T]}\|f''(u(t))\|_{L^\infty(\Omega)}
 \leq c t^{2\al-1},
\end{aligned}
\end{equation*}
where the last inequality follows from the fact that $u\in C^{\alpha}([0,T];C(\bar \Omega))$ and $\|u'(t)\|_{L^\infty(\Omega)}\le ct^{\alpha-1}$,
by Theorem \ref{thm:reg-u}.
Then the similar argument also yields that
\begin{equation}\label{eqn:wt}
\begin{aligned}
   \|w'(t)\|_{L^\infty(\Omega)}
\leq& c\Big(\|u''(t)\|_{L^\infty(\Omega)} \| u(t)-u_0 \|_{L^\infty(\Omega)}
  + \|u'(t)\|_{L^\infty(\Omega)}^2  \Big)  \max_{t\in[0,T]} \| f''(u(t)) \|_{L^\infty(\Omega)}\\
  \leq& ct^{2\alpha-2}.
\end{aligned}
\end{equation}
Then according to the  Sobolev--Slobodeckij seminorm \eqref{eqn:SS-seminorm},
we have for any $\epsilon\in(0,2\alpha)$
\begin{equation*}
\begin{aligned}
|  w |_{W^{2\alpha-\epsilon,1}(0,T;L^\infty(\Omega))} &= \int_{0}^{T}\hskip-5pt\int_{0}^{T} \frac{\|w(t)-w(s)\|_{L^\infty(\Omega)}}{|t-s|^{2\al+1-\epsilon}} \d t\d s=
\int_{0}^{T}\hskip-5pt\int_{0}^{T} \frac{\|\int_{s}^{t} w'(y)\d y\|_{L^\infty(\Omega)} }{|t-s|^{2\al+1-\epsilon}} \d t\d s\\
&\leq  \int_{0}^{T}\hskip-5pt\int_{0}^{T} \frac{|\int_{s}^{t} \| w'(y)\|_{L^\infty(\Omega)} \d y|}{|t-s|^{2\al+1-\epsilon}} \d t\d s.
\end{aligned}
\end{equation*}
Now by applying the estimate \eqref{eqn:wt}, we arrive at
\begin{equation*}
\begin{aligned}
|  w |_{W^{2\alpha-\epsilon,1}(0,T;L^\infty(\Omega))} & \leq c\int_{0}^{T}\hskip-5pt\int_{0}^{T} \frac{|\int_{s}^{t} y^{2\al-2}\d y|}{|t-s|^{2\al+1-\epsilon}} \d t\d s
 =c\int_{0}^{1}\hskip-5pt\int_{0}^{1} \frac{|\xi^{2\al-1}- \zeta^{2\al-1}|}{|\xi-\zeta|^{2\al+1-\epsilon}} \d \xi\d \zeta \\
& =c\bigg(\int_{0}^{1}\hskip-5pt\int_{0}^{\xi} \frac{\zeta^{2\al-1}-\xi^{2\al-1}}{(\xi-\zeta)^{2\al+1-\epsilon}} \d \zeta\d \xi
 +\int_{0}^{1}\hskip-5pt\int_{\xi}^{1} \frac{\xi^{2\al-1}-\zeta^{2\al-1}}{(\zeta-\xi)^{2\al+1-\epsilon}} \d \zeta\d \xi\bigg)\\
 &=2c\int_{0}^{1}\int_{0}^{\xi} \frac{\zeta^{2\al-1}-\xi^{2\al-1}}{(\xi-\zeta)^{2\al+1-\epsilon}} \d \zeta\d \xi
 =2c\int_{0}^{1}\xi^{-1+\epsilon}\d \xi \int_{0}^{1} \frac{ t^{2\al-1}-1}{(1-t)^{2\al+1-\epsilon}} \d t\\
&\le c_\epsilon \int_{0}^{1} \frac{ t^{2\al-1} -1 }{(1-t)^{2\al+1-\epsilon}} \d t.
\end{aligned}
\end{equation*}
Then the assertion that $w\in W^{1+2\al-\epsilon,1}(0,T;L^\infty(\Omega))$ follows from the  observation that
\begin{equation*}
\begin{aligned}
 \int_{0}^{1} \frac{ t^{2\al-1} -1 }{(1-t)^{2\al+1-\epsilon}} \d t
 &\le  \Big( \int_{0}^{\frac12}+  \int_{\frac12}^{1} \Big) \frac{ t^{2\al-1} -1 }{(1-t)^{2\al+1-\epsilon}} \d t \\
 &\le c + c \lim_{t\rightarrow 1} \frac{t^{2\al-1} -1 }{(1-t)^{2\al-\epsilon}} + c\int_{\frac12}^{1} \frac{ t^{2\al-2} }{(1-t)^{2\al-\epsilon}} \d t \le c.
\end{aligned}
\end{equation*}
\textbf{Case 2. $\alpha\in (1/2,1)$.}  In this case, using the estimate \eqref{eqn:wt}, it is easy to see that
$ \|\partial_t w (t)\|_{L^\infty(\Omega)} \in L^1(0,T)$. Then our aim is to show that
$$\partial_{tt}w \in {W^{2\alpha-\epsilon-1,1}(0,T;L^\infty(\Omega))}.$$
Using the expression for $\partial_{tt}w$ that
\begin{equation*}
 \partial_{tt} w (x,t)=(u'(x,t))^3 f'''(u(x,t))+2u'(x,t)u''(x,t)f''(u(x,t))+u'''(x,t) \int_{u_0(x)}^{u(x,t)}f''(\xi)\d \xi,
\end{equation*}
and Theorem \ref{thm:reg-u}, we derive that
\begin{equation}\label{eqn:wtt}
\begin{aligned}
  \|w_{tt}\|_{L^\infty(\Omega)}&\leq c \|(u'(t))\|_{L^\infty(\Omega)}^3
  +c\| u'(t)\|_{L^\infty(\Omega)}\|u''(t)\|_{L^\infty(\Omega)} \\
  &\quad +c\|u'''(t)\|_{L^\infty(\Omega)} \|u(t)-u_0\|_{L^\infty(\Omega)}\\
  &\leq c(t^{3\al-3}+t^{2\alpha-3}) \le ct^{2\alpha-3}.
\end{aligned}
\end{equation}
Recalling  the Sobolev--Slobodeckij seminorm  \eqref{eqn:SS-seminorm}, we have for any $\epsilon\in(0,2\alpha-1)$
\begin{equation*}
\begin{aligned}
\|  w_t \|_{W^{2\alpha-1-\epsilon,1}(0,T;L^\infty(\Omega))}
&= \int_{0}^{T}\hskip-5pt\int_{0}^{T} \frac{\|w_t(t)-w_s(s)\|_{L^\infty(\Omega)}}{|t-s|^{2\al-\epsilon}} \d t\d s\\
&= \int_{0}^{T}\hskip-5pt\int_{0}^{T} \frac{|\int_{s}^{t} \| \partial_{yy}w(y)\|_{L^\infty(\Omega)}\d y | }{|t-s|^{2\al-\epsilon}} \d t\d s.
\end{aligned}
\end{equation*}
Then we apply the estimate \eqref{eqn:wtt} and derive that
\begin{equation*}
\begin{aligned}
\|  w_t \|_{W^{2\alpha-1-\epsilon,1}(0,T;L^\infty(\Omega))}
&\leq c\int_{0}^{T}\hskip-5pt\int_{0}^{T} \frac{|\int_{s}^{t} y^{2\al-3}\d y|}{|t-s|^{2\al-\epsilon}} \d t\d s
 =c\int_{0}^{1}\hskip-5pt\int_{0}^{1} \frac{|\xi^{2\al-2}- \zeta^{2\al-2}|}{|\xi-\zeta|^{2\al -\epsilon}} \d \xi\d \zeta \\
& =c\bigg(\int_{0}^{1}\hskip-5pt\int_{0}^{\xi} \frac{\zeta^{2\al-2}-\xi^{2\al-2}}{(\xi-\zeta)^{2\al -\epsilon}} \d \zeta\d \xi
 +\int_{0}^{1}\hskip-5pt\int_{\xi}^{1} \frac{\xi^{2\al-2}-\zeta^{2\al-2}}{(\zeta-\xi)^{2\al -\epsilon}} \d \zeta\d \xi\bigg)\\
 &=2c\int_{0}^{1}\int_{0}^{\xi} \frac{\zeta^{2\al-2}-\xi^{2\al-2}}{(\xi-\zeta)^{2\al -\epsilon}} \d \zeta\d \xi
 =2c\int_{0}^{1}\xi^{-1+\epsilon}\d \xi \int_{0}^{1} \frac{ t^{2\al-2}-1}{(1-t)^{2\al -\epsilon}} \d t\\
 &\le c,
\end{aligned}
\end{equation*}
where the last inequality is a direct consequence of the fact that
\begin{equation*}
\begin{aligned}
 \int_{0}^{1} \frac{ t^{2\al-2} -1 }{(1-t)^{2\al -\epsilon}} \d t
 &\le  \Big( \int_{0}^{\frac12}+  \int_{\frac12}^{1} \Big) \frac{ t^{2\al-2} -1 }{(1-t)^{2\al -\epsilon}} \d t \\
 &\le c + c \lim_{t\rightarrow 1} \frac{t^{2\al-2} -1 }{(1-t)^{2\al+1-\epsilon}} + c\int_{\frac12}^{1} \frac{ t^{2\al-3} }{(1-t)^{2\al+1-\epsilon}} \d t \le c.
\end{aligned}
\end{equation*}
Therefore, we obtain that $\partial_{tt}w \in {W^{2\alpha-\epsilon-1,1}(0,T;L^\infty(\Omega))}$, and thus $u\in {W^{2\alpha+1-\epsilon,1}(0,T;L^\infty(\Omega))}$
for any $\alpha\in(1/2,1)$.

In conclusion, Case 1 and 2 together confirm the desired assertion for $\alpha\in(0,1/2)\cup(1/2,1)$. The critical case  $\alpha=1/2$ follows directly
from the result of Case 1 and hence the proof is completed.
\end{proof}

\section{Error analysis of modified BDF schemes}\label{sec:error}
The aim of this section is to present a complete error analysis for the high-order time stepping scheme \eqref{eqn:BDF-CQ-m}.
To begin with, we assume that the nonlinear term is globally Lipschitz continuous, i.e., there exists a constant $c_L$ such that
\begin{equation}\label{eqn:GL}
| f (s) -  f (t)| \le c_L |t-s| \qquad \text{for all}~~t,s\in\mathbb{R}.
\end{equation}
We shall establish numerical analysis under the assumption \eqref{eqn:GL}, and then extend the argument to the case without
that assumption.

\subsection{Existence and uniqueness of the time stepping solution}
In the analysis stated in the next subsection, we will always assume that the fully implicit scheme \eqref{eqn:BDF-CQ-m}
admits a unique solution. It is easy to confirm this assumption, provided that \eqref{eqn:GL} is valid.

In each time level, the fully implicit scheme \eqref{eqn:BDF-CQ-m}
 requires to solve a nonlinear elliptic problem
\begin{equation}\label{eqn:nonlinear}
  (b_0I + \tau^\alpha \Delta) v = w + \tau^\alpha   f(v)
\end{equation}
with  homogeneous Dirichlet boundary condition and some function $w\in C(\bar\Omega)$.
Next we show that there exists a unique solution to \eqref{eqn:nonlinear} in $C(\bar \Omega)$.
{By defining} the operator $M: C(\bar \Omega)\rightarrow C(\bar\Omega)$ as
$$M v = (b_0I + \tau^\alpha \Delta)^{-1} (w+\tau^\alpha f(v) ),$$
{and} applying the resolvent estimate \eqref{eqn:resol}, we observe that for any $v_1, v_2\in C(\bar\Omega)$
$$ \| M v_1 - M v_2\|_{L^\infty(\Omega)} =  \| (b_0\tau^{-\alpha}I + \Delta)^{-1}  (  f(v_1) -  f(v_2))  \|_{L^\infty(\Omega)} \le c c_L\tau^\alpha \| v_1-v_2 \|_{L^\infty(\Omega)}.$$
Then for $\tau$ small enough, $M$ is a contraction mapping, and hence there exists a unique $v\in C(\bar\Omega)$ such that $M(v)=v$, i.e.,
the nonlinear elliptic problem  \eqref{eqn:nonlinear} has a unique solution. As a result, we conclude that the fully implicit time stepping scheme
\eqref{eqn:BDF-CQ-m} admits a unique sequence of functions $\{ u_n \}_{n=1}^N$ via mathematical induction.



\subsection{Error analysis of the BDF scheme for linear problem}
The fundamental idea of error estimation is to apply the representation of the time stepping solution by contour integral in $\mathbb{C}$,
which has been extensively
used in existing studies \cite{LubichSloanThomee:1996, JinLazarovZhou:L1, jin2017analysis, JinLiZhou:correction, Yan:2018L1}.
We shall apply this technique to derive error estimates in $L^\infty(\Omega)$ norm. Note that the operator
$A$ defined in \eqref{eqn:A}
is self-adjoint, but not negative definite. However, the spectrum of $A$ has an upper bound,
since $f(u_0)\in L^\infty(\Omega)$. Now we define
\begin{equation}\label{eqn:la}
 \lambda= \max\big(1,\|f(u_0)\|_{L^\infty(\Omega)}\big),
\end{equation}
and observe that $L=\Delta + (f(u_0)-\lambda)I $ is self-adjoint and negative definite.
According to the resolvent estimate \eqref{eqn:resol} for $L$, we have the new resolvent estimate, for $v\in C(\bar\Omega)$
\begin{equation}\label{eqn:resol-2}
\begin{aligned}
  \|(z-A)^{-1}v\|_{ L^\infty(\Omega)}&=\|(z-\lambda -\big(\Delta + (f(u_0)-\lambda)I\big)v)^{-1}
  \|_{ L^\infty(\Omega)} \leq c_\phi|z-\lambda|^{-1} \|v \|_{ L^\infty(\Omega)},
  \end{aligned}
\end{equation}
for all
\begin{equation}\label{eqn:sig-la}
  z\in\Sigma_{\lambda,\phi}:=\{z\in\mathbb{C}\backslash\{\lambda\}
:|\arg (z-\lambda)|<\phi \}~~ \text{and}~~    \phi\in(\pi/2,\pi).
\end{equation}

To analyze the fully implicit BDF scheme, we shall start with the linear problem with a time-independent source term
\begin{equation} \label{eqn:v}
 \partial_t^\alpha v(t) - A v(t) = g_0 \quad\text{with}~~t\in(0,T],\quad \text{and} ~~ v(0)=u_0.
\end{equation}
Then the time stepping scheme reads
\begin{equation}\label{eqn:v-disc}
\left\{\begin{aligned}
&\bdalt (v-u_0)_n-Av_n=g_0+a_n^{(k)}(Au_0+g_0),\quad &1\leq n\leq k-1,\\
&\bdalt (v-u_0)_n-Av_n=g_0,\quad &k\leq n\leq N
\end{aligned}\right.
\end{equation}
with $v_0 = u_0$. The next lemma gives an estimate of the difference between $v(t_n)$ and $v_n$.
\begin{lemma}\label{lem:vv}
Let $v(t)$ and $v_n$ be the solutions of \eqref{eqn:v} and \eqref{eqn:v-disc}, respectively.
We assume that the conditions in Theorem \ref{thm:reg-u} hold true. Then there exists $\tau_0>0$,
such that for $\tau\le\tau_0$ the following error estimate holds
\begin{equation*}
\begin{aligned}
  \|v_n-v(t_n)\|_{L^\infty(\Omega)}
  & \leq c  \tau^k t_n^{\al-k}\|g_0 + Au_0\|_{L^\infty(\Omega)}
\end{aligned}
\end{equation*}
where the constant $c$ depends on $\alpha, k$ and $T$.
\end{lemma}\vskip5pt
\begin{proof}
Let $w(t)=v(t)-u_0$. Then the linear  problem \eqref{eqn:v} can be reformulated as
\begin{equation}\label{Eq:resubw}
  \dalt w-Aw=Au_0+g_0 \quad\text{with}~~t\in(0,T],\quad \text{and} ~~ w(0)=0.
\end{equation}
After taking Laplace transform, we derive that
\begin{equation*}
 \widehat{w}(z)= z^{-1}(z^\alpha-A)^{-1} (Au_0+g_0)
\end{equation*}
for any $z$ in the resolvent set of $A$. With inverse Laplace transform, the function $w(t)$ can be expressed explicitly by
\begin{equation*}
\begin{aligned}
  w(t) =\frac{1}{2\pi \i}\int_{\sigma_0-\i\infty}^{\sigma_0+\i\infty} e^{zt}K(z)(Au_0+g_0)\, \d z
\end{aligned}
\end{equation*}
with $\sigma_0$ such that $ (\sigma_0)^\alpha > \lambda$, where $\lambda$ is given in \eqref{eqn:la}  and the kernel $K(z)$ is defined by
\begin{equation*}\label{eqn:K}
  K(z)=z^{-1}(z^\al-A)^{-1}.
\end{equation*}
Now we deform the integral contour and obtain that
\begin{equation}\label{eqn:solrep-w}
  w(t) =\frac{1}{2\pi \i}\int_{\Gamma_{\theta,\sigma}} e^{zt}K(z)(Au_0+g_0)\, \d z
\end{equation}
where $\sigma\in(\lambda^{1/\alpha}, \sigma_0]$ and the contour $\Gamma_{\theta,\sigma}$ is defined by
\begin{equation}\label{eqn:Gamma}
  \Gamma_{\theta,\sigma}=\{z\in\C:  z=\sigma  + \rho e^{\pm \i\theta},\rho\geq 0\}\qquad \text{with any}~~ \theta\in(\pi/2,\pi),
\end{equation}
oriented with an increasing imaginary part.

Similarly, we may derive the integral representation of $v_n$ in the complex domain. By letting
$w_n = v_n - u_0$, we can reformulate the time stepping scheme as
\begin{equation}\label{eqn:w-disc}
\left\{\begin{aligned}
&\bdalt w_n-Aw_n=(1+a_n^{(k)})(Au_0+g_0),\quad &1\leq n\leq k-1,\\
&\bdalt w_n-Aw_n=Au_0+g_0,\quad &k\leq n\leq N
\end{aligned}\right.
\end{equation}
with $w_0 = 0$.
By multiplying $\xi^n$ on \eqref{eqn:w-disc} and taking summation over $n$
, we have
\begin{equation*}
  \sum_{n=1}^{\infty} \xi^n\bdalt w_n-\sum_{n=1}^{\infty} \xi^n Aw_n=\bigg(\sum_{n=1}^{\infty} \xi^n
  +\sum_{n=1}^{k-1} \xi^na_n^{(k)}\bigg)(Au_0+g_0).
\end{equation*}

For any given sequence $(f^n)_{n=0}^\infty$, let $\widetilde{f}(\xi):=\sum_{n=0}^{\infty}f^n\xi^n$ denote its generating function.
Since $w_0=0$, according to properties of discrete convolution, we have the identity
$$\sum_{n=1}^{\infty} \xi^n\bdalt w_n=\delta_\tau(\xi)^\al\widetilde{W}(\xi),$$
where $\delta_\tau(\xi)$ denotes the generating function of the standard BDF$k$ method \eqref{eqn:gen-k}. Therefore
\begin{equation*}
  (\delta_\tau(\xi)^\al -A)\widetilde{w}=\bigg(\frac{\xi}
  {1-\xi}+\sum_{n=1}^{k-1}\xi^na_n^{(k)}\bigg)(Au_0+g_0).
\end{equation*}
By the $A(\theta_k)$-stability of the BDF$k$ method \cite[pp. 251]{HairerWanner:1996},
for any $\xi$ such that $|\xi|=\rho\in(0,\frac12]$, there exists $\tau_0$ small enough such that $\delta_{\tau_0}(\frac12)^\alpha>\lambda+c_0$,
and we can find an angle $\theta_0\in(\pi/2,\pi)$ such that $\delta_\tau(\xi)^\alpha
\in \Sigma_{\lambda+c_0,\theta_0}$ for all $\tau \le \tau_0$
and hence the operator $(\delta_\tau(\xi)^\al -A)$ is invertible. Then
\begin{equation*}
  \widetilde{w}(\xi)=K(\delta_\tau(\xi))\tau^{-1}\mu(\xi)(Au_0+g_0),
\end{equation*}
where $\mu(\xi)=\delta(\xi)(\frac{\xi}{1-\xi}+\sum_{n=1}^{k-1}\xi^na_n^{(k)}).$

Let $\rho\in(0,\frac12]$ and $\tau\le \tau_0$, it is easy to see that
$\widetilde{w}(\xi)$ is analytic with respect to $\xi$ in the circle $|\xi|=\rho$
on the complex plane, then with the change of variables $\xi=e^{-z\tau}$ and Cauchy's integral formula, we have the following expression
\begin{equation}\label{eqn:solrep-wn}
\begin{aligned}
  w_n&=\frac{1}{2\pi\i}\int_{|\xi|=\rho} \xi^{-n-1}\widetilde{w}(\xi) \d\xi\\
  &=\frac{1}{2\pi\i}\int_{\Gamma^\tau} e^{zt_n}K(\delta_\tau(e^{-z\tau}))\mu(e^{-z\tau})(Au_0+g_0)\d z\\
  &=\frac{1}{2\pi\i}\int_{\Gamma^\tau_{\theta,\sigma_0}} e^{zt_n}K(\delta_\tau(e^{-z\tau}))\mu(e^{-z\tau})(Au_0+g_0)\d z,
\end{aligned}
\end{equation}
where $\Gamma^\tau:=\{z=-\ln(\rho)/\tau+\i y: y\in\mathbb{R},\,|y|\leq \pi/\tau\}$ and
$\Gamma^\tau_{\theta,\sigma}=\{z\in\Gamma_{\theta,\sigma}:|\Imag(z)|\leq \pi/\tau\}$ with $\sigma= -\ln(\frac12)/\tau_0$.
The deformation of contour from $\Gamma^\tau$ to $\Gamma^\tau_{\theta,\sigma_0}$
in the last equation is achieved due to the analyticity and periodicity of the function
$e^{zt_n}K(\delta_\tau(e^{-z\tau}))\mu(e^{-z\tau})$. Then there exists $\theta\in(\pi/2,\pi)$ close to $\pi/2$ such that
$\delta_\tau(e^{-z\tau})^\alpha \in \Sigma_{\lambda+c_0,\theta_0+\epsilon}$ for some small $\epsilon>0$.

Now we  recall the
properties of the generating function $\delta_\tau(\xi)$ and correction term $\mu(\xi)$, which have already been established in
\cite[eq. (2.13) and Theorem B.1.]{JinLiZhou:correction}. In particular, in case that $z\in \Gamma^\tau_{\theta,\sigma}$ there holds that
\begin{equation}\label{eqn:app}
  \begin{aligned}
& c_1|z|\leq
|\delta_\tau(e^{-z\tau})|\leq c_2|z|,
  &&|\delta_\tau(e^{-z\tau})-z|\le c\tau^k|z|^{k+1},\\
 & |\delta_\tau(e^{-z\tau})^\alpha-z^\alpha|\leq c\tau^k|z|^{k+\alpha},
&& |\mu(e^{-z\tau})-1| \leq c \tau^k |z|^k.
\end{aligned}
\end{equation}

To derive an estimate for $w_n - w(t_n)$, we compare those two solution representations \eqref{eqn:solrep-w} and  \eqref{eqn:solrep-wn}.
To this end, we use the splitting
\begin{equation}\label{eqn:split}
\begin{aligned}
u_n-u(t_n)&=w_n -w(t_n)\\
&=\frac{1}{2\pi\i}\int_{\Gamma^\tau_{\theta,\sigma}} e^{zt_n}(K(\delta_\tau(e^{-z\tau}))\mu(e^{-z\tau})-K(z))(Au_0+g_0)\d z\\
&\quad-\frac{1}{2\pi \i}\int_{\Gamma_{\theta,\sigma}\backslash \Gamma^\tau_{\theta,\sigma}} e^{zt_n}K(z)(Au_0+g_0)\d z:=I-II.
\end{aligned}
\end{equation}
Next, we shall bound these two terms separately.
By the resolvent estimate \eqref{eqn:resol-2} and approximation properties  \eqref{eqn:app}, we have
\begin{equation}\label{eqn:KK}
\begin{aligned}
   &\|[K(\delta_\tau(e^{-z\tau}))-K(z)]\psi\|_{L^\infty(\Omega)} \\
\leq& |\delta_\tau(e^{-z\tau})^{-1}-z^{-1}|\|(\delta_\tau(e^{-z\tau})^\al-A)^{-1}\psi\|_{L^\infty(\Omega)}\\
&+ |z|^{-1}\|[(\delta_\tau(e^{-z\tau})^\al-A)^{-1}-(z^\al-A)^{-1}]\psi\|_{L^\infty(\Omega)}\\
=&c\tau^k|z|^{k-1}\|(\delta_\tau(e^{-z\tau})^\al-A)^{-1}\psi \|_{L^\infty(\Omega)}\\
  &+ |z|^{-1} |z^\al-\delta_\tau(e^{-z\tau})^\al|
  \|(\delta_\tau(e^{-z\tau})^\al-A)^{-1} (z^\al-A)^{-1}\psi\|_{L^\infty(\Omega)}\\
  &\leq c\tau^k|z|^{k-1}|\delta_\tau(e^{-z\tau})^\al-\lambda|^{-1}(1+|z|^{ \al }|z^\al-\lambda|^{-1}) \| \psi \|_{L^\infty(\Omega)}
\end{aligned}
\end{equation}
for $\psi \in C(\bar\Omega)$. For any $z=\sigma+\rho e^{i\theta}$ with $\rho\le1$, we have the uniform bound that
\begin{equation}\label{eqn:smallrho}
 |\delta_\tau(e^{-z\tau})^\al-\lambda|^{-1} + |z^\al-\lambda|^{-1}\le c
\end{equation}
since $\delta_\tau(e^{-z\tau})^\al \in \Sigma_{\lambda+c_0,\theta_0+\epsilon}$. Besides, for $z=\sigma+\rho e^{i\theta}$ with $\rho>1$,
it holds that
\begin{equation}\label{eqn:largerho-1}
 |z^\al-\lambda|^{-1} \le c|z^\alpha|^{-1}  \le c\rho^{-\alpha},
\end{equation}
and similarly, using the fact that $\delta_\tau(e^{-z\tau})^\al \in   \Sigma_{\lambda+c_0,\theta_0+\epsilon}$ and the
approximation properties of generating functions in \eqref{eqn:app},
we have for $\rho>1$
\begin{equation}\label{eqn:largerho-2}
 |\delta_\tau(e^{-z\tau})^\al-\lambda|^{-1} \le c| \delta_\tau(e^{-z\tau})^\al|^{-1} \le c |z|^{-\alpha} \le c\rho^{-\alpha}.
\end{equation}
The same argument also gives the same bound for $z=\sigma+\rho e^{-i\theta}$. Now for the first term in \eqref{eqn:split},
we have
\begin{equation*}
  \begin{aligned}
 \|I\|_{L^\infty(\Omega)}&=\|\frac{1}{2\pi\i}\int_{\Gamma^\tau_{\theta,\sigma}} e^{zt_n}(K(\delta_\tau(e^{-z\tau}))\mu(e^{-z\tau})-K(z))(Au_0+g_0)\d z \|_{L^\infty(\Omega)}\\
&\leq \|\frac{1}{2\pi\i}\int_{\Gamma^\tau_{\theta,\sigma}} e^{zt_n}K(\delta_\tau(e^{-z\tau}))(\mu(e^{-z\tau})-1)(Au_0+g_0)\d z \|_{L^\infty(\Omega)}\\
&\quad + \|\frac{1}{2\pi\i}\int_{\Gamma^\tau_{\theta,\sigma}} e^{zt_n}(K(\delta_\tau(e^{-z\tau}))-K(z))(Au_0+g_0)\d z \|_{L^\infty(\Omega)}=: I_1 + I_2.%
\end{aligned}
\end{equation*}
The term $I_1$ can be bounded using estimates \eqref{eqn:app}, \eqref{eqn:smallrho} and \eqref{eqn:largerho-1}
\begin{equation*}
  \begin{aligned}
 I_1 &\le c\tau^k\|Au_0+g_0\|_{L^\infty(\Omega)}  \int_{\Gamma^\tau_{\theta,\sigma}} e^{\Real(z)t_n}|z|^{k-1}|z^\al-\lambda|^{-1}|\d z|\\
  &\le c e^{\sigma t_n}\tau^k\|Au_0+g_0\|_{L^\infty(\Omega)} \Big(  \int_0^1 e^{-c \rho t_n}  \d \rho +  \int_1^\infty e^{-c \rho t_n}\rho^{k-1-\alpha} \d \rho \Big)\\
  &\le c_T \tau^k t_n^{\alpha-k}\|Au_0+g_0\|_{L^\infty(\Omega)}.
\end{aligned}
\end{equation*}
Similarly, we apply \eqref{eqn:app}-\eqref{eqn:largerho-2} to derive a proper bound for $I_2$
\begin{equation*}
  \begin{aligned}
 I_2 &\le c\tau^k\|Au_0+g_0\|_{L^\infty(\Omega)}  \int_{\Gamma^\tau_{\theta,\sigma}}
 e^{\Real(z)t_n}|z|^{k-1} |\delta_\tau(e^{-z\tau})^\al-\lambda|^{-1}(1+|z|^{ \al }|z^\al-\lambda|^{-1})|\d z|\\
  &\le c e^{\sigma t_n} \tau^k\|Au_0+g_0\|_{L^\infty(\Omega)} \Big(  \int_0^1 e^{-c \rho t_n} \d \rho +  \int_1^\infty e^{-c \rho t_n}\rho^{k-1 -\alpha} \d \rho \Big)\\
    &\le c_T \tau^k t_n^{\alpha-k}\|Au_0+g_0\|_{L^\infty(\Omega)}.
\end{aligned}
\end{equation*}
Finally, we bound the second term in \eqref{eqn:split} by using the resolvent estimate \eqref{eqn:resol-2}:
\begin{equation*}
\begin{aligned}
  \|II\|_{L^\infty(\Omega)}
  &\leq c  e^{\sigma t_n}\|Au_0+g_0\|_{L^\infty(\Omega)}\int_{\pi/(\tau\sin\theta)}^{\infty}
  e^{-c\rho t_n }\rho^{-1}|(\sigma + \rho e^{\i\theta})^\al-\lambda|^{-1}\d \rho\\
  &\leq  c_T\tau^k\|Au_0+g_0\|_{L^\infty(\Omega)}\int_{\pi/(\tau\sin\theta)}^{\infty}
  e^{-c\rho t_n }\rho^{k-1-\alpha} \d \rho\quad(\text{since } 1\leq \tau^k|z|^k)\\
  &\leq c_T\tau^kt_n^{\al-k}\|Au_0+g_0\|_{L^\infty(\Omega)}.
\end{aligned}
\end{equation*}
This completes the proof of the lemma.
\end{proof}

\begin{remark}\label{eqm:const}
The generic constant $c$ in Lemma \ref{lem:vv}
depends on the terminal time $T$ with $c(T)\sim O(e^{\sigma T})$ for some $\sigma>0$. Therefore,
the error estimate in Lemma \ref{lem:vv} is not uniform in $T$ and hence it is not suitable for long-time estimate.
This is because the operator $A=\Delta + f'(u_0)  I$ might not be negative definite,
and hence the solution might blow up exponentially as $T\rightarrow \infty$.
In case that $A$ is negative definite, we can obtain an error estimate which is uniform in large terminal time
(e.g., \cite{LiWangZhou:2020, JinLiZhou:correction}).
\end{remark}

Now we turn to the subdiffusion problem driven by a general source term:
\begin{equation} \label{eqn:w-conti}
 \partial_t^\alpha w(t) - A w(t) = g(t) \quad\text{with}~~t\in(0,T],\quad \text{and} ~~ w(0)=0,
\end{equation}
whose time stepping scheme reads
\begin{equation}\label{eqn:w-discere}
 \begin{aligned}
 \bdalt w_n-Aw_n=g_n:=g(t_n),\quad &1\leq n\leq k-1, \quad \text{with}\quad w_0=0.
\end{aligned}
\end{equation}
The time stepping solution $w_n$ can be represented by a discrete convolution 
\begin{equation}\label{eqn:disc-con}
w_n = \tau \sum_{j=1}^n  E_\tau^{n-j}g_j, ~~\text{where}~~
E_{\tau}^n  = \frac{1}{2\pi\mathrm{i}}\int_{\Gamma_{\theta,\sigma}^\tau } e^{zt_n} ({ \delta_\tau(e^{-z\tau})^\alpha}+A )^{-1}\,\d z.
\end{equation}
The angle $\theta$ and parameter $\sigma$ are chosen as those in the proof of Lemma \ref{lem:vv}.
Then by \eqref{eqn:app}, \eqref{eqn:smallrho} and \eqref{eqn:largerho-2}, we derive that
\begin{equation}\label{eqn:est-E}
 \begin{aligned}
\| E_{\tau}^n \psi\|_{L^\infty(\Omega)} &= \bigg\|\frac{1}{2\pi\mathrm{i}}\int_{\Gamma_{\theta,\sigma}^\tau } e^{zt_n} ({ \delta_\tau(e^{-z\tau})^\alpha}+A)^{-1}\psi\,\d z \bigg\|_{L^\infty(\Omega)} \\
&\le ce^{\sigma T} \Big(\int_{1}^{\frac{\pi}{\tau\sin\theta}} e^{-c\rho t_n } \rho^{-\alpha} \d \rho
+  \int_{0}^1 e^{-c\rho t_n}   \d \rho\Big)
\le c_T (t_n+\tau)^{\alpha-1} .
\end{aligned}
\end{equation}
Therefore, it holds the stability that
\begin{equation}\label{eqn:stab}
\|w_n\|_{L^\infty(\Omega)}  \le c_T\Big(\tau \sum_{j=1}^n t_{n-j+1}^{\alpha-1} \| g_j \|_{L^\infty(\Omega)}\Big).
\end{equation}
Here we assume that the source term $g$ satisfies certain compatibility condition, e.g.,
$$g^{(j)}(0)=0, \qquad j=0,1,2,\ldots,k-1.$$
For such a source term $g$, by using the resolvent estimates \eqref{eqn:resol}
and the technique in the proof of  Lemma \ref{lem:vv},
the estimate of $w_n-w(t_n)$ can be done similarly (hence omitted) as that given
in \cite[Lemma 3.7]{jin2017analysis}, i.e., for all $\ell=1,2,\ldots,k$
\begin{equation*}
 \begin{aligned}
\| w(t_n) - w_n \|_{L^\infty(\Omega)} &\le c_T \tau^\ell \int_0^{t_n}(t_n-s)^{\alpha-1} \| g^{(\ell)}(s) \|_{L^\infty(\Omega)}\,\d s \\
&\le c_T  \tau^\ell  \Big(t_n^{\alpha-1} \| g \|_{W^{\ell,1}((0,t_n/2);L^\infty(\Omega))} + t_n^\alpha  \| g \|_{C^{\ell}([t_n/2,t_n];L^\infty(\Omega))} \Big).
\end{aligned}
\end{equation*}
Then by the interpolation, we have the following estimate.\vskip5pt

\begin{lemma}\label{lem:source}
Suppose that $g\in W^{\ell+s,1}((0,T);C(\bar\Omega))\cap C^{\ell+1}((0,T);C(\bar\Omega))$ and $g^{(j)}=0$ with $\ell\in \mathbb{N}^+$, $s\in(0,1)$ and $j=0,1,\ldots,\ell$.
Let $w(t)$ and $w_n$ be the solutions of \eqref{eqn:w-conti} and \eqref{eqn:w-discere}, respectively.
Then the following error estimate holds
\begin{equation*}
\| w(t_n) - w_n  \|_{L^\infty(\Omega)} \le c t_n^{\alpha-1} \tau^{\min(k,\ell+s)},
\end{equation*}
where the constant $c$ depends only on $\alpha,~g$ and $T$.
\end{lemma}\vskip5pt

\subsection{Error analysis of the BDF scheme for nonlinear problem}
Now we turn to the error estimate for the fully implicit scheme \eqref{eqn:BDF-CQ-m}.
To this end, we begin with
the following lemma, which provides a discrete H\"older bound of the time stepping solution to  \eqref{eqn:BDF-CQ-m}.
\begin{lemma}\label{lem:sol-bound}
Assume that the same conditions in Theorem \ref{thm:reg-u} holds valid and further $f$ satisfies \eqref{eqn:GL}.
Let  $\{u_n\}_{n=1}^N$ be the solution to the time stepping scheme \eqref{eqn:BDF-CQ-m}.
Then we have
\begin{equation*}
  \max_{1 \le n\le N} \| u_n \|_{L^\infty(\Omega)} +  \max_{1 \le n\le N} t_n^{-\alpha} \| u_n - u_0 \|_{L^\infty(\Omega)} \le c , \qquad \text{for}~~n=1,2,\ldots,N,
\end{equation*}
where the constant $c$ depends on $T, u_0,  f$ but is independent of $\tau$ and $N$.
\end{lemma}
\begin{proof}
By the preceding argument, we have the following representation of $u_n$:
\begin{equation}\label{eqn:disc-con2}
u_n = u_0 + F_\tau^{n} (Au_0 +g_0) + \tau \sum_{j=1}^n  E_\tau^{n-j} R(u_j; u_0)
\end{equation}
as well as the bound that (by Theorem \ref{lem:vv} and the estimate \eqref{eqn:est-E})
$$ \| F_\tau^{n} \|_{C(\bar\Omega)\rightarrow C(\bar\Omega)} \le c_T t_n^\alpha\quad\text{and}\quad  \| E_\tau^{n}\|_{C(\bar\Omega)\rightarrow C(\bar\Omega)} \le c_T t_n^{\alpha-1}. $$
Therefore, by the Lipchitz continuity of the modified potential term $\bar f(s)$, we have
\begin{equation}\label{eqn:disc-con3}
\begin{aligned}
 \| u_n - u_0 \|_{L^\infty(\Omega)} &\le c_T t_n^\alpha + \tau \sum_{j=1}^n  t_{n-j+1}^{\alpha-1} \| R(u_j; u_0)\|_{L^\infty(\Omega)} \\
 &\le c_T t_n^\alpha + \tau \sum_{j=1}^n  t_{n-j+1}^{\alpha-1} \Big(\| f(u_j) -  f(u_0)\|_{L^\infty(\Omega)} + \|   f'(u_0) (u_j-u_0)\|_{L^\infty(\Omega)}\Big) \\
 &\le c_T t_n^\alpha + \tau \sum_{j=1}^n  t_{n-j+1}^{\alpha-1} \| u_j-u_0 \|_{L^\infty(\Omega)}.
\end{aligned}
\end{equation}
Then by the discrete Gr\"onwall's inequality \cite[Lemma 7.1]{Elliott:1992}, we obtain that
$$   \| u_n - u_0 \|_{L^\infty(\Omega)} \le c  t_n^{\alpha}, $$
where the constant $c$ depends on $T, u_0,  f$, but it is independent of $\tau$ and $N$.
Finally, the uniform bound of $\| u_n \|_{L^\infty(\Omega)}$ follows from the triangle inequality.
\end{proof}\vskip5pt

Now we are ready to state our main theorem in the section.
\begin{theorem}\label{thm:error}
Assume that the same conditions in Theorem \ref{thm:reg-u} holds valid and further $f$ satisfies \eqref{eqn:GL}.
Let $u(t)$ be the solution of the semilinear subdiffusion problem \eqref{Eqn:fde}
and $\{u_n\}_{n=1}^N$ be the solution of fully implicit scheme \eqref{eqn:BDF-CQ-m}. Then the following error
estimate holds
\begin{equation*}
\| u_n - u(t_n)   \|_{L^\infty(\Omega)} \le c \tau^{\min(k,1+2\alpha-\epsilon)} t_n^{\alpha -\min (k, 1+2\alpha-\epsilon)},
\end{equation*}
for any $t_n > 0$ and arbitrarily small $\ep>0$. Here the constant $c$ depends on $T, u_0,  f, \ep$, but it is independent of $\tau$ and $N$.
\end{theorem}

\begin{proof}
To begin with, we split the solution $u(t)$ into two components
$$ u(t) = v(t) + w(t), $$
where $v$ is the solution of \eqref{eqn:v}, and $w$ satisfies \eqref{eqn:w-conti} with $g=R(u;u_0)$.
Similarly, the time stepping solution can also be separated by
$$ u_n = v_n + w_n,  $$
where $v_n$ is the solution of \eqref{eqn:v-disc}, and $w_n$ satisfies \eqref{eqn:w-discere} with $g_n=  R(u_n;u_0)$, i.e., by \eqref{eqn:disc-con}
$$ w_n = \tau\sum_{j=1}^n E_\tau^{n-j}  R(u_j;u_0). $$
We note that the difference between $v(t_n)$ and $v_n$ has been estimated in Lemma \ref{lem:vv}. In order to study $w_n - w(t_n)$, we use an
intermediate solution $\bar w_n$ which satisfies \eqref{eqn:w-discere} with {$g_n=R(u(t_n);u_0)$} and can be represented by
$$ \bar w_n = \tau\sum_{j=1}^n E_\tau^{n-j} R(u(t_j);u_0).$$
Then the regularity of $R(u;u_0)$ proved in Theorem \ref{thm:reg-R} and the estimate in Lemma \ref{lem:source} yield
\begin{equation*}
\| \bar w_n - w(t_n)   \|_{L^\infty(\Omega)} \le c \tau^{1+2\alpha-\epsilon} t_n^{\ep-\alpha-1}.
\end{equation*}
To sum up, we derive a bound of $e_n =u_n - u(t_n)$:
\begin{align*}
\|e_n\|_{L^\infty(\Omega)} &=  \|  v_n - v(t_n) \|_{L^\infty(\Omega)} + \|\bar w_n - w(t_n) \|_{L^\infty(\Omega)}  + \|  w_n - \bar w_n \|_{L^\infty(\Omega)}  \\
&\le c \tau^k t_n^{\alpha-k} + c \tau^{1+2\alpha-\epsilon} t_n^{\alpha-1} + \tau\sum_{j=1}^n \| E_\tau^{n-j} [R(u_j;u_0)- R(u(t_j);u_0)]\|_{L^\infty(\Omega)}\\
&\le c  \tau^{\min(k,1+2\alpha-\epsilon)} t_n^{\alpha -\min (k, 1+2\alpha-\epsilon)} + c\tau\sum_{j=1}^n t_{n-j+1}^{\alpha-1} \|  R(u_j;u_0)- R(u(t_j);u_0)\|_{L^\infty(\Omega)}.
\end{align*}
Recalling Lemma \ref{lem:sol-bound} and the fact that $u\in C^\alpha([0,T];C(\bar\Omega))$, we obtain that
\begin{equation*}
\begin{aligned}
 &\quad \| R(u_j;u_0)- R(u(t_j);u_0)\|_{L^\infty(\Omega)} \\
 &\le \Big\|   \int_{u_0}^{u(t_j)} (u(t_j)-u_j) f''(s)\,ds \Big\|_{L^\infty(\Omega)}
+ \Big\| \int_{u (t_j)}^{u_j} (u_j - s)   f''(s) \,ds \Big\|_{L^\infty(\Omega)}\\
&\le c \Big(\|u_j - u_0\|_{L^\infty(\Omega)} + \|  u(t_j)- u_0\|_{L^\infty(\Omega)} \Big) \| e_n \|_{L^\infty(\Omega)} \\
&\le c t_n^\alpha \| e_n \|_{L^\infty(\Omega)},
\end{aligned}
\end{equation*}
and hence we arrive at the estimate
\begin{align*}
\|e_n\|_{L^\infty(\Omega)}
&\le c  \tau^{\min(k,1+2\alpha-\epsilon)} t_n^{\alpha -\min (k, 1+2\alpha-\epsilon)} + c\tau\sum_{j=1}^n t_{n-j+1}^{\alpha-1} t_j^\alpha \| e_j \|_{L^\infty(\Omega)}.
\end{align*}
After multiplying $t_n^\alpha$ on both sides, we have
\begin{align*}
t_n^\alpha \|e_n\|_{L^\infty(\Omega)}
&\le c \tau^{\min(k,1+2\alpha-\epsilon)} t_n^{2\alpha -\min (k, 1+2\alpha-\epsilon)} + ct_n^\alpha \tau\sum_{j=1}^n t_{n-j+1}^{\alpha-1} t_j^\alpha \| e_j \|_{L^\infty(\Omega)}.
\end{align*}
Noting that $2\alpha -\min (k, 1+2\alpha-\epsilon)>-1$,
we  apply the Gr\"onwall's inequality \cite[Lemma 7.1]{Elliott:1992} for $t_n^\alpha \|e_n\|_{L^\infty(\Omega)}$ and derive that
\begin{align*}
t_n^\alpha \|e_n\|_{L^\infty(\Omega)} \le c \tau^{\min(k,1+2\alpha-\epsilon)} t_n^{2\alpha -\min (k, 1+2\alpha-\epsilon)}.
\end{align*}
This completes the proof.
\end{proof}\vskip5pt

\begin{remark}\label{rem:bfisf-2}
The result in Theorem \ref{thm:error} implies a uniform-in-time error
$$ \max_{1\le n\le N} \| u_n - u(t_n)  \|_{L^\infty(\Omega)} \le c \tau^{\alpha-\epsilon}, $$
for some small $\ep>0$. This result is consistent with the error estimate in \cite{JinLiZhou:nonlinear}.
\end{remark}\vskip5pt

\begin{remark}\label{rem:bfisf-3}
The error estimate in Theorem \ref{thm:error} indicates that the best convergence rate of the corrected BDF$k$ scheme \eqref{eqn:BDF-CQ-m}
is almost of order $O(\tau^{\min(k,1+2\alpha)})$, due to the low regularity of the remainder $R(u;u_0)$ (see Theorem \ref{thm:reg-R} and Lemma \ref{lem:source}).
The reason is that $u$ is nonsmooth in the time direction,
even though the initial condition is smooth and compatible with the boundary condition.
This phenomena contrasts sharply with its normal parabolic counterpart, i.e., $\alpha=1$.
For instance, it has been proved in \cite{CrouzeixThomee} that the time stepping schemes of the semilinear parabolic equation
are able to achieve a better convergence rate in case of regular initial data.
\end{remark}

\subsection{Numerical analysis without globally Lipschitz condition}\label{ssec:local-lip}
The preceding analysis could be easily extended to
the nonlinear subdiffusion problem without the globally Lipschitz condition \eqref{eqn:GL}. For completeness, we briefly sketch
the argument in this section.

Under the assumptions in Theorem \ref{thm:reg-u} and letting
$$b = \| u \|_{L^\infty((0,T)\times\Omega)} + 1,$$
we are able to define a smooth function $\bar{f}$
such that
\begin{equation}\label{eqn:GL-2-0}
\bar f(s) = f(s)\qquad \text{for all}~~-b\le s \le b,
\end{equation}
and it is globally Lipschitz continuous
\begin{equation}\label{eqn:GL-2}
|\bar f (s) - \bar f (t)| \le c_L |t-s| \qquad \text{for all}~~t,s\in\mathbb{R}.
\end{equation}
Then we consider the BDF scheme  with potential term $\bar f$ instead of $f$
\begin{equation}\label{eqn:BDF-CQ-m-2-r}
\left\{\begin{aligned}
&\bdalt (u-u_0)_n-\Delta u_n= a_n^{(k)}(\Delta u_0 + f(u_0)) + \bar f(u_n),\quad &&1\leq n\leq k-1,\\
&\bdalt (u-u_0)_n-\Delta u_n= \bar f(u_n),\quad &&k\leq n\leq N.
\end{aligned}\right.
\end{equation}

Then under the condition \eqref{eqn:GL-2} we know that \eqref{eqn:BDF-CQ-m-2-r} admits a unique solution.
Meanwhile, Theorem \ref{thm:error} indicates a uniform-in-time error estimate
$$ \max_{1\le n\le N} \| u_n - u(t_n)  \|_{L^\infty(\Omega)} \le c \tau^{\alpha-\epsilon}, $$
for some small $\ep>0$, where the constant $c$ depends on $T, u_0, \bar f, \ep$.

As a result, for $\tau < \tau_0$ such that $c \tau_0^{\alpha-\epsilon}=1$, we have
$$ \max_{1\le n\le N}\| u_n \|_{L^\infty(\Omega)} \le \max_{1\le n\le N} \|  u(t_n)  \|_{L^\infty(\Omega)} + c\tau^{\alpha-\epsilon} \le b+1. $$
Therefore $\bar f(u_n) = f(u_n)$ for all $1\le n\le N$, and the modified time stepping scheme \eqref{eqn:BDF-CQ-m-2-r}
is identical to the original one \eqref{eqn:BDF-CQ-m-r} (or equivalently \eqref{eqn:BDF-CQ-m}).
Then we have the following corollary.\vskip5pt

\begin{corollary}\label{cor:error}
Assume that the same conditions in Theorem \ref{thm:reg-u} hold valid.
Let $u(t)$ be the solution of the semilinear subdiffusion problem \eqref{Eqn:fde}
and $\{u_n\}_{n=1}^N$ be the solution of fully implicit scheme \eqref{eqn:BDF-CQ-m}. Then the following error
estimate holds
\begin{equation*}
\| u_n - u(t_n)   \|_{L^\infty(\Omega)} \le c \tau^{\min(k,1+2\alpha-\epsilon)} t_n^{\alpha -\min (k, 1+2\alpha-\epsilon)},
\end{equation*}
for any $t_n > 0$ and arbitrarily small $\ep>0$. Here the constant $c$ depends on $T, u_0,  f, \ep$, but it is independent of $\tau$ and $N$.
\end{corollary}\vskip5pt

\begin{remark}\label{rem:comp-CLP}
In \cite{CuestaLubichPalencia:2006}, Cuesta et. al studied a second-order BDF method for solving a related (but different) subdiffusion model
\begin{equation}\label{eqn:fde2}
 u - \partial_t^{-\alpha} \Delta u = u_0 + \partial_t^{-1} f(u)\quad \text{with}\quad u(0)=u_0,
\end{equation}
under the assumption that $f$ is sufficiently smooth and the solution $u$ can be expanded as
\begin{equation}\label{eqn:exp}
 u(t) = \sum_{m,l\ge0;m+l\alpha<2} c_{m,l} t^{m+l\alpha} + v(t) \quad \text{with}\quad c_{ml}\in D(\Delta)~~ \text{and}~~
 v\in C^2([0,T];D(\Delta)).
\end{equation}
This assumption requires stronger compatibility conditions of $u_0$ and $f(u_0)$. As a simple example, we consider the
homogeneous problem, i.e. $f\equiv0$. In this case, the solution of the subdiffusion problem \eqref{Eqn:fde}
can be expanded by Mittag-Leffler function as
\begin{equation*}
 u(t) = E_{\alpha,1}(\Delta t^\alpha) u_0 = \sum_{k=0}^\infty \frac{1}{\Gamma(\alpha k+1)} \Big((\Delta)^k u_0\Big) t^{\alpha k}.
\end{equation*}
Then assumption \eqref{eqn:exp} requires that $u_0\in D((-\Delta)^{1+2/\alpha})$ which is stronger than what we assumed in
this work.
\end{remark}

\section{Fully discrete scheme and error analysis}\label{sec:fully}
In this section, we will briefly discuss the fully discrete scheme for solving the nonlinear subdiffusion equation \eqref{Eqn:fde}.
We shall start with a spatially semidiscrete scheme for problem \eqref{Eqn:fde} based on the standard
Galerkin finite element method (see e.g., \cite{Jin:2013, JinLazarovZhou:overview} for linear subdiffusion problems).

For $h\in(0,h_0]$, $h_0>0$, we denote
by $\mathcal{T}_h = \{K_j\}$ a triangulation of $\Omega_h=\text{Int}(\cup \overline K_j)$
into mutually disjoint open face-to-face simplices $K_j$. Assume that all vertices of a simplex $K_j$
locate on $\partial\Omega$. We also assume that
$\{\mathcal{T}_h\}$  is globally quasi-uniform, i.e., $|K_j|\ge c h^d$ with a given $c>0$.
Let $X_h$ be the finite dimensional space of continuous piecewise
linear functions associated with $\mathcal{T}_h$, that vanish outside $\Omega_h$. 
Then we define the $L^2(\Omega)$ projection $P_h:L^2(\Omega)\to X_h$ and Ritz projection $R_h:H_0^1\rightarrow X_h$  respectively by
\begin{equation*}
\begin{split}
     (P_h \varphi,v_h) &=(\varphi,v_h) , \quad \forall\, v_h\in X_h,\\
     (\nabla R_h v,\nabla v_h) &= (\nabla v,\nabla v_h),\quad \forall\, v_h\in X_h.
\end{split}
\end{equation*}

The semidiscrete scheme reads:  find $u_h(t) \in X_h$ such that
\begin{equation}\label{eqn:semi-FEM}
(\partial_t^\alpha u_h (t), v_h) + (\nabla u_h(t), \nabla v_h) = (f(u_h(t)),v_h)\quad\text{for all}~v_h\in X_h,
\end{equation}
with $u_h(t) = R_h u_0$.
Let $\Delta_h:X_h\rightarrow X_h$ denote
the Galerkin finite element approximation of the Dirichlet Laplacian $\Delta$, defined by
$$(\Delta_hw_h,v_h):=-(\nabla w_h,\nabla v_h),\quad \forall\, w_h,v_h\in X_h .$$
Then the spatially semidiscrete scheme \eqref{eqn:semi-FEM} could be written as
\begin{equation}\label{eqn:semi-FEM-r}
\partial_t^\alpha u_h (t) - \Delta_h u_h(t) = P_h f(u_h), \quad \text{with}~ u_h(0)=R_hu_0.
\end{equation}
With the Laplace transform and convolution rule, $u_h(t)$ can be explicitly expressed by
\begin{equation}\label{eqn:sol-rep-semi}
  u(t)=(I + F_h(t)\Delta_h) u_0+\int_{0}^{t} E_h(t-s)f(u(s))\d s,
\end{equation}
where the operators {$F_h(t)$} and {$E_h(t)$} are defined by
\begin{equation}\label{eqn:op-semi}
  F_h(t)=\frac{1}{2\pi\i}\int_{\Gamma_{\theta,\delta}} e^{zt}z^{-1}(z^\al-\Delta_h )^{-1}\d z\quad\text{and}\quad
  E_h(t)=\frac{1}{2\pi\i}\int_{\Gamma_{\theta,\delta}} e^{zt}(z^\al-\Delta_h )^{-1}\d z,
\end{equation}
respectively. 
Recall that the discrete Laplacian satisfies the resolvent estimate
in $L^\infty\II$ sense (cf. \cite[Theorem 1.1]{Bakaev:2003}), i.e., for any angle $\phi\in(\pi/2,\pi)$,
\begin{align}\label{eqn:reg-disc-inf}
\|  (z-\Delta_h)^{-1} w_h \|_{L^\infty \II } \le c |z|^{-1} \| w_h \|_{L^\infty \II} \quad  \forall ~ z \in\Sigma_\phi.
\end{align}
This immediately implies the following smoothing properties:
\begin{align}\label{eqn:reg-disc-op-01}
\|F_h\Delta_h v_h\|_{L^\infty\II} + t^{1-\alpha}\| E_h v_h\|_{L^\infty\II} + t \| E_h \Delta_h v_h\|_{L^\infty\II}  \le c\|v_h\|_{L^\infty\II}\quad \forall~v_h\in X_h,
\end{align}
which plays an important role in error analysis. Note that the $L^\infty\II$-norm error analysis of the scheme \eqref{eqn:semi-FEM-r}
remains scarce, even though the  $L^2\II$-norm estimate has been completely understood (cf. \cite{Karaa:nonlinear, JinLiZhou:nonlinear}).
For completeness, we shall provide an error estimate in $L^\infty\II$-norm.

\subsection{Spatially semidiscrete scheme for the linear problem}
First we recall some error estimates for the following linear subdiffusion equation:
\begin{equation}\label{PDEv-linear}
\partial_t^\alpha v(t)-\Delta  v(t)=g(t), \quad\,\,\forall t\in (0,T] ,
\end{equation}
where $g$ is a given source function, and  $v(0)\in D$ is the given initial condition.
The semidiscrete FEM for \eqref{PDEv-linear}
seeks $v_h(t)\in X_h$ such that
\begin{equation}\label{eqn:semidiscrete-linear}
\partial_t^\alpha v_h(t)-\Delta_h v_h(t)=P_hg(t), \quad\,\,\forall t\in (0,T],
\end{equation}
with $v_h(0)=R_hv(0)$. 
Recall that $R_h$ has the almost stability property \cite[eq. (6.60)]{Thomee:2006}
\begin{equation}\label{eqn:ritz-stab}
\| R_h w \|_{L^\infty\II} \le c\ell_h \|  w\|_{L^\infty\II},\quad \text{with}~~\ell_h=\max(1, \log(1/h)).
\end{equation}
 To derive the error estimate of \eqref{PDEv-linear}, we need the following lemma for
the Ritz projection $R_h$, where the proof relies on the smoothing property of the solution operator $F(t)$:
\begin{align}\label{eqn:reg-F}
\|  \Delta F(t) w \|_{L^p \II} \le c  \|   w \|_{L^p \II}, \quad \text{for all} ~~p\in[1,\infty).
\end{align}
This follows directly from the representation \eqref{eqn:op} and the resolvent estimate \cite[Theorem 3.1]{Ouhabaz:1995}
$$  \|  (z - \Delta)^{-1} w \|_{L^p\II}  \le c_p |z|^{-1}\|  w \|_{L^p\II}\quad \forall~~ z \in\Sigma_\phi,~\phi\in(\pi/2,\pi), ~p\in [1,\infty).$$

\begin{lemma}\label{lem:space-error-linear-0}
Let $v$ be the solution of  the linear problem \eqref{PDEv-linear}. Then there holds 
\begin{align*}
 \| (v-R_hv)(t)  \|_{L^\infty\II} \le ch^2\ell_h^2 \big(\| \Delta v(0) \|_{L^\infty\II} + \int_0^t \| g'(s) \|_{L^\infty\II}\,ds\big)
\end{align*}
with $\ell_h=\max(1, \log(1/h))$.
\end{lemma}
\begin{proof}
Let $I_h$ be the Lagrange interpolation operator. Then we have $$v-R_hv = (R_h  - I)(v - I_h v)$$
and hence by \eqref{eqn:ritz-stab} and the approximation property of $I_h$, we derive for $2 \le p<\infty $
$$ \| v-R_hv  \|_{L^\infty} \le c \ell_h \| v - I_h v \|_{L^\infty\II} \le c h^{2-2/p}\ell_h  \|  v \|_{W^{2,p}\II}.$$
Now using the full elliptic regulariy, we have  for $2 \le p<\infty $ \cite[eq. (6.78)]{Thomee:2006}
$$\|  v \|_{W^{2,p}\II} \le c p \| \Delta v \|_{L^p\II}.$$
Recalling the solution representation \eqref{eqn:sol-rep}, we have
\begin{equation*}
\begin{split}
 \Delta v&=  \Delta (I + F(t)\Delta) v(0) + \int_0^t \Delta E(t-s) g(s)\,\d s\\
 &= \Delta (I + F(t)\Delta) v(0) + \int_0^t \Delta F(t-s) g'(s)\,\d s -   \Delta (F(0) g(t)  -  F(t)g(0))\\
 \end{split}
\end{equation*}
Now we apply the smoothing property \eqref{eqn:reg-F} and arrive at
\begin{align*}
 \| \Delta v\|_{L^p\II} 
 &\le c\| \Delta v(0) \|_{L^p\II} + c\int_0^t \| g'(s) \|_{L^p\II}\,ds.
 \end{align*}
 Then the desired result follows immediately by choosing $p=\ell_h$.
 \end{proof}

The semidiscrete solution $v_h$ satisfies the following error estimate.
\begin{lemma}[Semidiscrete solution of linear problems]\label{lem:space-error-linear}
For the semidiscrete solution $v_h$ to problem \eqref{eqn:semidiscrete-linear},
{there holds, with $\ell_h=\max(1, \log(1/h))$, that}
\begin{align*}
\max_{t\in[0,T]}
\|v_h(t)-v(t)\|_{L^2(\Omega)} \leq ch^2 \ell_h^3  \big(\| \Delta v(0) \|_{L^\infty\II} + \int_0^t \| g'(s) \|_{L^\infty\II}\,ds\big).
\end{align*}
\end{lemma}
\begin{proof}
 We use the splitting
 $v_h -v  = (v_h - P_h v ) + (P_h v - v) = : \psi +  \theta.$
 By Lemma \ref{lem:space-error-linear-0} and \cite[Corollary]{Douglas:1974}, it is easy to see for all $t\in[0,T]$
\begin{align}\label{eqn:the}
 \| \theta(t) \|_{L^\infty \II} + \| (P_h v- R_h v)(t) \|_{L^\infty \II}  \le c h^2\ell_h^2 \big(\| \Delta v(0) \|_{L^\infty\II} + \int_0^t \| g'(s) \|_{L^\infty\II}\,\d s\big).
\end{align}
 Besides, we note that $\psi$ satisfies the equation
\begin{align*}
\partial_t^\alpha \psi (t) - \Delta_h \psi (t) = \Delta_h (R_h-P_h) v(t), \quad \text{with}~\psi(0) = (R_h - P_h)v.
\end{align*}
Therefore, by the representation \eqref{eqn:sol-rep-semi}, we arrive at
\begin{align*}
 \psi (t)  &= (I+F_h(t)\Delta_h)(R_h - P_h)v(0) + \int_0^t E_h(t-s) \Delta_h (R_h-P_h) v(s)\,ds
 =: I_1+I_2.
\end{align*}
The estimate of $I_1$  follows directly from \eqref{eqn:reg-disc-op-01} and \eqref{eqn:the}
\begin{align*}
\| I_1 \|_{L^\infty\II} &\le c \|  (R_h - P_h)v(0) \|_{L^\infty\II} \le  c  h^2\ell_h^2  \| \Delta v(0) \|_{L^\infty\II} .
\end{align*}
For the second term, we apply the inverse inequality for finite element functions,
as well as \eqref{eqn:reg-disc-op-01} and \eqref{eqn:the}, to obtain that
\begin{align*}
\| I_2 \|_{L^\infty\II} &\le c h^{-2\epsilon} \int_0^t (t-s)^{-1+\epsilon} \| (R_h-P_h) v(s) \|_{L^\infty\II}\,\d s\\
&\le c \epsilon^{-1} h^{2-2\epsilon} \ell_h^2  \big(\| \Delta v(0) \|_{L^\infty\II} + \int_0^t \| g'(s) \|_{L^\infty\II}\,\d s\big)
\end{align*}
by choosing $ \epsilon=1/\ell_h$, then we complete the proof of the lemma.
\end{proof}

\subsection{Error analysis for the nonlinear problem}
Now we turn to the nonlinear problem \eqref{Eqn:fde}. The following lemma provides an error estimate
of the semidiscrete scheme \eqref{eqn:semi-FEM-r}.
\begin{lemma}\label{lem:error-semi-2}
Assume that the same conditions in Theorem \ref{thm:reg-u} {hold} valid.
Then the semidiscrete problem \eqref{eqn:semi-FEM-r} has a unique solution
$u_h\in C([0,T]\times\bar\Omega)$, which satisfies
\begin{align}\label{error-estimate-FEM}
 \max_{0\leq t\leq T}\|u(t)-u_h(t)\|_{L^\infty(\Omega)}\le c h^2\ell_h^3,\quad \text{with}~~\ell_h=\max(1, \log(1/h)).
\end{align}
\end{lemma}
\begin{proof}
To begin with, we assume that the nonlinear term $f:\mathbb{R}\rightarrow\mathbb{R}$ is Globally Lipschitz continuous.
Then, by the argument in \cite[Theorem 3.1]{JinLiZhou:nonlinear}, the existence and uniqueness of the solution $u_h$ hold.
It remains to establish the estimate \eqref{error-estimate-FEM}.
To this end, we define $v_h(t)$ as the solution of
\begin{align*}
    \partial_t^\al v_h(t) - \Delta_h v_h(t) = P_h f(u(t)), \quad \text{with}\quad v_h(0) =R_h u_0.
\end{align*}
This together with Lemma \ref{lem:space-error-linear} and Theorem \ref{thm:reg-u} yields the following estimate for $t \ge 0$
\begin{equation}\label{eqn:wh}
\begin{split}
\| (u-v_h)(t) \|_{L^2(\Omega)}
    \le ch^2\ell_h^3.
\end{split}
\end{equation}
Meanwhile, we note that $\rho_h:=v_h-u_h$ satisfies the following equation
\begin{equation*}
    \partial_t^\al \rho_h(t) -\Delta \rho_h(t) = P_h f (u(t)) - P_h f(u_h(t)), \quad \text{with}\quad \rho_h(0)=0.
\end{equation*}
Then, by the smoothing property \eqref{eqn:reg-disc-op-01}, the Lipschitz continuity of $f$ and the stability of $P_h$ in $L^\infty\II$ \cite{Douglas:1974}, we derive that
\begin{equation*}
\begin{split}
    \|\rho_h(t)\|_{L^\infty} &\le \int_0^t \|  E_h(t-s) P_h [f (u(s)) -   f(u_h(s)) ]  \|_{L^\infty\II}\,\d s \\
    &\le c \int_0^t (t-s)^{\alpha-1}\|  P_h [f (u(s)) -   f(u_h(s)) ]  \|_{L^\infty\II}\,\d s\\
    &\le c \int_0^t (t-s)^{\alpha-1}\|  u(s) -   u_h(s)  \|_{L^\infty\II}\,\d s\\
&\le ch^2\ell_h^3+ c \int_0^t (t-s)^{\alpha-1} \|  \rho_h(s)  \|_{L^\infty\II}\,\d s.
\end{split}
\end{equation*}
Then by the Gr\"onwall's inequality, we have
\begin{equation*}
\max_{t\in[0,T]}\| \rho_h(t) \|_{L^2\II} \le c h^2\ell_h^3.
\end{equation*}
This and \eqref{eqn:wh} directly imply the desired result.
Then the same argument as the one in Section \ref{ssec:local-lip} helps to remove the
the globally Lipschitz condition.
\end{proof}

Finally, we consider the fully discrete scheme: find $U_h^n$ such that
\begin{equation}\label{eqn:BDF-CQ-m-fully}
\left\{\begin{aligned}
&\bdalt (U_h^n-U_h^0)-\Delta_h U_h^n= a_n^{(k)}(\Delta_h U_h^n + f(U_h^0) + f(U_h^0),\quad &&1\leq n\leq k-1,\\
&\bdalt (U_h^n-U_h^0)-\Delta_h U_h^n=f(U_h^n),\quad &&k\leq n\leq N.
\end{aligned}\right.
\end{equation}
Then by the resolvent estimate \eqref{eqn:reg-disc-inf}, all the {arguments} in Sections \ref{sec:prelim} and \ref{sec:error}
{work} for the spatially discrete problems \eqref{eqn:semi-FEM-r} and \eqref{eqn:BDF-CQ-m-fully}. Therefore we have the following
corollary.

\begin{corollary}\label{cor:error-fully-0}
Assume that the same conditions in Theorem \ref{thm:reg-u} {hold} valid.
Let $u_h(t)$ be the solution of the semidiscrete scheme \eqref{eqn:semi-FEM-r}
and $\{U_h^n\}_{n=1}^N$ be the solution of fully discrete scheme \eqref{eqn:BDF-CQ-m-fully}. Then the following error
estimate holds
\begin{equation*}
\| U^n_h - u_h(t_n)   \|_{L^\infty(\Omega)} \le c \tau^{\min(k,1+2\alpha-\epsilon)} t_n^{\alpha -\min (k, 1+2\alpha-\epsilon)},
\end{equation*}
for any $t_n > 0$ and arbitrarily small $\ep>0$. Here the constant $c$ depends on $T, u_0,  f, \ep$, but it is independent of $h$, $\tau$ and $N$.
\end{corollary}\vskip5pt

This corollary together with Lemma \ref{lem:error-semi-2} immediately leads to the error estimate of the fully discrete scheme \eqref{eqn:BDF-CQ-m-fully}.
\begin{theorem}\label{thm:error-fd}
Assume that the same conditions in Theorem \ref{thm:reg-u} {hold} valid.
Let $u(t)$ be the solution of the semilinear subdiffusion problem \eqref{Eqn:fde}
and $\{U_h^n\}_{n=1}^N$ be the solution of fully discrete scheme \eqref{eqn:BDF-CQ-m-fully}. Then for $\ell_h=\max(1,\log({1/h}))$, the following error
estimate holds
\begin{equation*}
\| U_h^n - u(t_n)   \|_{L^\infty(\Omega)} \le c h^2\ell_h^3+ \tau^{\min(k,1+2\alpha-\epsilon)} t_n^{\alpha -\min (k, 1+2\alpha-\epsilon)},
\end{equation*}
for any $t_n > 0$ and arbitrarily small $\ep>0$. The constant $c$ depends on $T, u_0,  f, \ep$, but it is independent of $h$, $\tau$ and $N$.
\end{theorem}

\section{Numerical experiments}\label{sec:numerics}
In this section, we present numerical results to illustrate and support our theoretical findings.
We consider the nonlinear subdiffusion model with $\Omega=(0,1)^2$
\begin{equation}\label{Eqn:fde-numeric}
\left\{\begin{aligned}
\dalt u-\frac1{10}\Delta u&=4(u-u^3)&&\text{ in } \Omega\times (0,T),\\
u &=0 &&\text{ on  }  \pt\Omega\times (0,T),\\
u(0)&=u_0 &&\text{ in }\Omega ,
\end{aligned}\right.
\end{equation}

In the computation, we divided the domain $\Omega$ into regular right triangles with $M$ equal subintervals of length $h$
on each side of the domain. The numerical solutions are computed by using fully discrete scheme \eqref{eqn:BDF-CQ-m-fully}.
 In each step, we solved the nonlinear elliptic problem by Newton's iteration.
We fixed the  spatial mesh size $h=1/100$,  {computed the numerical solution $\{ U_{h}^N \}$ with temporal step size
$\tau=T/N$ with $T=1$, $N=100\times 2^\ell$, $\ell=0,1,\ldots,4$} and reported
\begin{align*}e_\tau= \big\|U_h^N-u_h(t_N)\big\|_{L^\infty\II}.\end{align*}
Since the semidiscrete solution $u_h$ is unavailable,
we compute reference solutions on a finer mesh, i.e., the fully discrete solution $U_h^N$ with
$h=1/100$, $N=20000$ and $k=6$.

We consider the following problem data:
\begin{align*} u_0(x,y) = 4x(1-x)y(1-y), \end{align*}
where the initial condition satisfies
\begin{align*}u_0, \Delta u_0 \in C(\bar\Omega)\qquad \text{and}\qquad u=0~~\text{on}~\partial\Omega.\end{align*}
Therefore, our assumptions on initial condition (i.e., $u_0\in D$) are fulfilled.
In Table \ref{Tab:casea}, we present numerical results of the corrected $k$-step BDF scheme \eqref{eqn:BDF-CQ-m}.
Numbers in brackets are the theoretical convergence rates. Numerical results show that the convergence rate is $O(\tau^{\min(k,1+2\alpha)})$.
For example, in case that $\alpha=0.7$, we observe an $O(\tau^{2.4})$ rate of BDF$k$ scheme with $k=3,4,5,6$, but an $O(\tau^2)$ rate in case that $k=2$.
This is in good agreement with our theoretical results.
In Table \ref{Tab:casea-2}, we present numerical results for uncorrected  $k$-step BDF schemes \eqref{TD-scheme}. We observe that
all schemes are first-order accurate. This phenomena has already been reported for the linear fractional evolution equations \cite{Jin:SISC2016, JinLiZhou:correction}.
This implies the  necessity of the modification in the starting steps.

\begin{table}[h!]
  \centering
    \caption{Corrected BDF$k$ scheme \eqref{eqn:BDF-CQ-m} at {$T=1$} with $h=1/100$ and  $\tau=1/(100\times 2^\ell)$ }\label{Tab:casea}
    \vskip-5pt
  \begin{tabular}{|c|c|ccccc|c|}
  \hline
  $\al$ & $k\backslash \ell$ & 0 & 1 & 2 & 3 & 4 & rate\\ \hline
        & $k=2$ &2.94e-06  &9.99e-07  &3.45e-07  &1.20e-07  &4.19e-08   & $\approx$ 1.52 (1.60) \\
        & $k=3$ &2.43e-06  &8.90e-07  &3.21e-07  &1.14e-07  &3.99e-08   & $\approx$ 1.51 (1.60) \\
  0.3   & $k=4$ &4.36e-06  &1.57e-06  &5.58e-07  &1.96e-07  &6.82e-08   & $\approx$ 1.52 (1.60) \\
        & $k=5$ &9.73e-06  &3.45e-06  &1.21e-06  &4.23e-07  &1.46e-07   & $\approx$ 1.53 (1.60) \\
        & $k=6$ &5.17e-09  &1.70e-09  &5.60e-10  &1.85e-10  &6.09e-11   & $\approx$ 1.60 (1.60) \\
\hline
        & $k=2$ &2.79e-06  &7.53e-07  &2.02e-07  &5.44e-08  &1.45e-08   & $\approx$ 1.91(2.00) \\
        & $k=3$ &6.42e-07  &1.75e-07  &4.63e-08  &1.20e-08  &3.07e-09   & $\approx$ 1.97 (2.00) \\
  0.5   & $k=4$ &8.63e-07  &2.24e-07  &5.77e-08  &1.47e-08  &3.74e-09   & $\approx$ 1.98 (2.00) \\
        & $k=5$ &1.52e-06  &3.93e-07  &1.01e-07  &2.57e-08  &6.53e-09   & $\approx$ 1.98 (2.00) \\
        & $k=6$ &8.57e-09  &2.15e-09  &5.38e-10  &1.34e-10  &3.37e-11   & $\approx$ 2.00 (2.00) \\
\hline
        & $k=2$ &3.13e-06  &7.88e-07  &1.98e-07  &4.97e-08  &1.25e-08   & $\approx$ 2.00 (2.00) \\
        & $k=3$ &8.57e-08  &1.97e-08  &4.13e-09  &8.31e-10  &1.63e-10   & $\approx$ 2.35 (2.40) \\
  0.7   & $k=4$ &1.05e-07  &1.99e-08  &3.79e-09  &7.20e-10  &1.37e-10   & $\approx$ 2.39 (2.40) \\
        & $k=5$ &1.55e-07  &2.97e-08  &5.66e-09  &1.08e-09  &2.05e-10   & $\approx$ 2.39 (2.40) \\
        & $k=6$ &1.05e-08  &2.00e-09  &3.78e-10  &7.18e-11  &1.36e-11   & $\approx$ 2.40 (2.40) \\
\hline
\end{tabular}
\end{table}

\begin{table}[h!]
  \centering
    \caption{Uncorrected BDF$k$ scheme \eqref{TD-scheme} at {$T=1$} with $h=1/100$ and  {$\tau=1/(100\times 2^\ell)$ }}\label{Tab:casea-2}
    \vskip-5pt
  \begin{tabular}{|c|c|ccccc|c|}
  \hline
  $\al$ & $k\backslash \ell$ & 0 & 1 & 2 & 3 & 4 & rate\\ \hline
        & $k=2$  &6.01e-05  &2.99e-05  &1.49e-05  &7.47e-06  &3.73e-06   & $\approx$ 1.00 (1.00) \\
        & $k=3$  &5.99e-05  &2.99e-05  &1.49e-05  &7.46e-06  &3.73e-06   & $\approx$ 1.00 (1.00) \\
  0.3   & $k=4$ &5.99e-05  &2.99e-05  &1.49e-05  &7.45e-06  &3.73e-06   & $\approx$ 1.00 (1.00) \\
        & $k=5$  &5.98e-05  &2.99e-05  &1.49e-05  &7.45e-06  &3.72e-06   & $\approx$ 1.00 (1.00) \\
        & $k=6$  &9.72e-06  &4.85e-06  &2.43e-06  &1.21e-06  &6.06e-07   & $\approx$ 1.00 (1.00) \\
\hline
        & $k=2$ &1.05e-04  &5.24e-05  &2.61e-05  &1.31e-05  &6.53e-06   & $\approx$ 1.00 (1.00) \\
        & $k=3$ &1.05e-04  &5.23e-05  &2.61e-05  &1.31e-05  &6.53e-06  & $\approx$ 1.00 (1.00) \\
  0.5   & $k=4$ &1.05e-04 &5.22e-05  &2.61e-05  &1.31e-05  &6.53e-06   & $\approx$ 1.00 (1.00) \\
        & $k=5$ &1.05e-04  &5.22e-05  &2.61e-05  &1.31e-05  &6.53e-06   & $\approx$ 1.00 (1.00) \\
        & $k=6$ &3.85e-05  &1.92e-05  &9.62e-06  &4.81e-06  &2.40e-06   & $\approx$ 1.00 (1.00) \\
\hline
        & $k=2$ &1.60e-04  &8.00e-05  &3.99e-05  &2.00e-05  &9.97e-06   & $\approx$ 1.00 (1.00) \\
        & $k=3$ &1.60e-04  &7.98e-05  &3.99e-05  &1.99e-05  &9.97e-06   & $\approx$ 1.00 (1.00) \\
  0.7   & $k=4$ &1.60e-04  &7.98e-05 &3.99e-05 &1.99e-05  &9.97e-06   & $\approx$ 1.00 (1.00) \\
        & $k=5$ &1.60e-04  &7.98e-05  &3.99e-05  &1.99e-05  &9.97e-06   & $\approx$ 1.00 (1.00) \\
        & $k=6$ &1.10e-04  &5.50e-05  &2.75e-05  &1.38e-05  &6.88e-06   & $\approx$ 1.00 (1.00) \\
\hline
\end{tabular}
\end{table}

\section*{Acknowledgements}
The authors are grateful to Prof. Buyang Li for his suggestion and valuable comments on an earlier version of the paper.

\bibliographystyle{abbrv}

\end{document}